\documentclass[11pt,a4paper,reqno]{amsart}
\usepackage[latin1]{inputenc}
\usepackage[english]{babel}
\usepackage{amsmath}
\usepackage{amsfonts}
\usepackage{amssymb}
\usepackage{mathrsfs}
\usepackage{latexsym}
\usepackage{natbib}

\usepackage{mathrsfs}
\usepackage{mathtools}
\usepackage[margin=2.5cm]{geometry}
\usepackage{enumerate}

\usepackage{graphicx,color}
\usepackage{caption}
\usepackage{subcaption}
\usepackage[percent]{overpic}

\usepackage[plainpages=false,colorlinks,hyperindex,bookmarksopen,linkcolor=red,citecolor=blue,urlcolor=blue]{hyperref}

\bibpunct{[}{]}{;}{n}{,}{,}

\newtheorem{theorem}{Theorem}[section]

\newtheorem{corollary}{Corollary}
\newtheorem{proposition}{Proposition}[section]
\newtheorem{remark}{Remark}[section]

\numberwithin{equation}{section}

\allowdisplaybreaks

\author{Mirko D'Ovidio}
        \address{Department of Basic and Applied Sciences for Engineering\newline
Sapienza University of Rome\newline via A. Scarpa 10, Rome, Italy}

\author{Francesco Iafrate}
        \address{Department of Basic and Applied Sciences for Engineering\newline
Sapienza University of Rome\newline via A. Scarpa 10, Rome, Italy}
       \email[Corresponding author]{francesco.iafrate@uniroma1.it}

\title{Elastic drifted Brownian motions and non-local boundary conditions}

\begin{document}
\date{\today}

\maketitle

\begin{abstract}
We provide a deep connection between elastic drifted Brownian motions and inverses to tempered subordinators. Based on this connection, we establish a link between multiplicative functionals and dynamical boundary conditions given in terms of non-local equations in time. Indeed, we show that the multiplicative functional associated to the elastic Brownian motion with drift is equivalent to a multiplicative functional associated with non-local boundary conditions of tempered type. By exploiting such connections we write some functionals of the drifted Brownian motion in terms of a simple (positive and non-decreasing) process, the inverse of tempered subordinator. In our view, such a representation is useful in many applications and brings new light on dynamic boundary value problems.
\end{abstract}


\section{Introduction}

In this paper we focus on elastic drifted Brownian motions and their governing equations equipped with fractional boundary conditions of the form
\begin{align}
\label{genFBC}
\mathfrak{D}^\Phi_t \varphi(t, 0) + c_1\, \varphi(t, 0) = c_2
\end{align}
where the constants $c_1, c_2$ will be better specified and $\mathfrak{D}^\Phi_t$ is a non-local operator characterized by the Bernstein symbol $\Phi$. The constants $c_1,c_2$ and the symbol $\Phi$ depend on the drift. In particular, for $\lambda\geq 0$, $\Phi(\lambda) = \sqrt{\lambda+\eta} - \sqrt{\eta}$ where $\eta\geq 0$ will be written in terms of the drift of the elastic Brownian motion. This symbol $\Phi$ introduces the so-called tempered fractional derivative with tempering parameter $\eta$ (see Section \ref{sec:temp}).

A first relevant fact is that the tempered fractional derivative turns out to be strictly related with the infinitesimal generator of the drifted Brownian motion. However, the condition \eqref{genFBC} is more than a surprising relation involving this generator. Indeed, we provide a deep connection between the time-dependent boundary condition and the multiplicative functional associated with the elastic drifted Brownian motion. In particular, \eqref{genFBC} is associated with an equivalent multiplicative functional which is written in terms of tempered subordinators and their inverses. This entails a deep relation also  between subordinators and functionals of the Brownian motion.

A family $M=\{M_t\}_{t\geq 0}$ of real-valued random variable is called multiplicative functional (of a given Markov process) provided: $i)$ $M_t$ is progressively measurable; $ii)$ $M_{t+s}=M_t (M_s \circ \theta_t)= M_s (M_t \circ \theta_s)$ a.s. for each $t,s\geq 0$ ($\theta_\alpha$ is the translation operator); $iii)$ $0 \leq M_t \leq 1$ for all $t\geq 0$ (\cite[Chapter III]{BluGet68}). It is well known that  two multiplicative functionals are equivalent if and only if they generate the same semigroup (\cite[Proposition 1.9]{BluGet68}). In particular, the multiplicative functional uniquely characterizes the semigroup (\cite[Theorem 3.3]{BluGet68}). 

Let us consider the drift $\pm \mu$ with $\mu\geq 0$. For the elastic drifted Brownian motion $\widetilde{X}^{\pm\mu} =\{\widetilde{X}^{\pm \mu}_t \}_{t\geq 0}$ on $[0, \infty)$ we can write
\begin{align}
\label{semigIntro}
\mathbf{E}_x[f(\widetilde{X}^{\pm \mu}_t)] = \mathbf{E}_x[f(X^{\pm \mu}_t)\, M^{\pm \mu}_t]
\end{align}
where $X^{\pm\mu} = \{X^{\pm \mu}_t\}_{t\geq 0}$ is a reflecting Brownian motion on $[0, \infty)$ and $M^{\pm\mu}_t$ is the multiplicative functional associated with the elastic condition. Formula \eqref{semigIntro} gives the probabilistic representation of the solution to
\begin{equation}
\label{CP}
\begin{cases}
\displaystyle  \frac{\partial u}{\partial t} = \pm \mu \frac{\partial u}{\partial x} + \frac{\partial^2 u}{\partial x^2}, & x \geq 0,\, t>0\\
\displaystyle u(0, x) = f(x), & x\geq 0
\end{cases}, \quad \mu \geq 0
\end{equation}
with (elastic) boundary condition
\begin{align}
\label{BCIntro}
\frac{\partial u}{\partial x} (t,0) = c\, u(t,0), \quad t>0
\end{align}
where $ c> 0 $.
The problem to find a probabilistic representation for the solution to \eqref{CP} and \eqref{genFBC} can be addressed as in \cite{FBVP2dov} via time change. Non-local boundary value problems can be considered as useful models for motions on trap domains (with irregular boundaries). The solutions to the problems \eqref{CP}-\eqref{BCIntro} ans \eqref{CP}-\eqref{genFBC} obviously differ except for a constant initial datum $f$. Here, we are interested in the equivalence between \eqref{genFBC} and \eqref{BCIntro} for the Cauchy problem \eqref{CP}. Thus, we focus on the lifetime of $\widetilde{X}^{\pm\mu}$ and we provide some connections between the fractional boundary condition and the corresponding process, that is a non negative and non decreasing process which is an inverse to a tempered subordinator with symbol $\Phi$ with tempering parameter $\eta = (\pm \mu/2)^2$.

\subsection{Main results and plan of the work}

First we provide some deep relations between elastic drifted Brownian motion and an inverse to a tempered subordinator. Then, we prove equivalence between the multiplicative functionals $M$ and $\overline{M}$ where the latter is written in terms of an inverse to a tempered subordinator. This permits a very fruitful change between the elastic drifted Brownian motion and a non-decreasing process (the inverse to a subordinator) in studying the problem \eqref{CP} - \eqref{BCIntro}.

In Section \ref{secDerSub} and Section \ref{sec:temp} we introduce the tempered subordinator $H$ and its inverse $L$ together with non-local operators in time. In section \ref{sec:EDBM} we introduce the elastic drifted Brownian motion and the following equivalences in law (Theorem \ref{thm0} and Theorem \ref{thm00})
\begin{align}
\label{deepRelations}
\max_{0\leq s \leq t} X^\mu_s \stackrel{d}{=} L_t, \quad t\geq 0 \quad \textrm{and} \quad \max_{0\leq s \leq t} X^{-\mu}_s \stackrel{d}{=} L_t \wedge T_\mu, \quad t\geq 0 
\end{align}
where $T_\mu$ is an independent exponential random variable. 

In Section \ref{helpfulSection} we discuss an intuitive example in case of zero drift. This corresponds to the case of stable subordinator (indeed $\eta=0$) and therefore, the Caputo derivative is involved. 

In Section \ref{sec:caseI} we confirm the relations discussed above in \eqref{deepRelations} in terms of boundary value problems. Indeed, the process $L_t$ (associated to the problem \eqref{probell}) is related with $X^{\pm \mu}$ (in terms of formulas \eqref{deepRelations}) as well as the condition \eqref{BCIntro} is related with \eqref{genFBC} in the domain $D(G_{\pm \mu})$. In particular, we are concerned with the solution 
\begin{align*}
u(t,x)= \mathbf{E}_x [\overline{M}_t^{\pm \mu}]
\end{align*}
to the problem 
\begin{equation*}
\begin{cases}
\displaystyle  \frac{\partial u}{\partial t} = \pm \mu \frac{\partial u}{\partial x} + \frac{\partial^2 u}{\partial x^2}, & x \geq 0,\, t>0\\
\displaystyle u(0, x) = \mathbf{1}(x\geq 0), & x\geq 0
\end{cases}
,\qquad \mu \geq 0
\end{equation*}
with the boundary condition \eqref{genFBC} where $c_1, c_2$ characterize the interplay between the inverse to a subordinator and the exponential random variable. The non-local boundary condition \eqref{genFBC} takes the following forms:
\begin{itemize}
\item[-] In case of positive drift $\mu$ (Theorem \ref{thmI}, Theorem \ref{thmIbis} and Theorem \ref{last-thmI}),
\begin{align}
\label{FBCIntro}
\mathfrak{D}^{\frac{1}{2},\mu}_t u(t, 0) + (\mu + c)\,  u(t,0) = \mu, \quad t>0;
\end{align}
\item[-] In case of negative drift $-\mu$ (Theorem \ref{thmII} and Theorem \ref{last-thmII}),
\begin{align}
\label{FBCIntroo}
\mathfrak{D}^{\frac{1}{2},\mu}_t u(t, 0) + c\,  u(t,0) = 0, \quad t>0;
\end{align}
\item[-] In case of zero drift $\mu=0$ (as a by-product of the previous theorems),
\begin{align}
\label{FBCIntrooo}
\mathfrak{D}^{\frac{1}{2}}_t u(t, 0) + c\,  u(t,0) = 0, \quad t>0,
\end{align}
\end{itemize}  
where $\mathfrak{D}^{\frac{1}{2},\mu}_t $ denotes the tempered Caputo derivative defined in Section \ref{sec:temp}. First we show that \eqref{BCIntro} is equivalent to \eqref{FBCIntro} and \eqref{FBCIntroo}. That is, $\overline{M}^{\pm\mu}_t$ is equivalent to the multiplicative functional $M^{\pm\mu}_t$ associated with the elastic condition for a drifted Brownian motion. Moreover, we show that $\overline{M}^{\pm \mu}_t$ can be written in terms of inverses to tempered stable subordinators $L_t$ and the exponential r.v. $T_\mu$ for which $\mathbf{P}(T_\mu >x)=e^{-\mu x}$ and $T_0 = \infty$ with probability $1$. We show that  $c_2=\mu$ in \eqref{FBCIntro} introduces $T_\mu$ with $\mu>0$ whereas, $c_2=0$ in \eqref{FBCIntroo} introduces $T_0$. If $\mu=0$, then we obtain the elastic Brownian motion with elastic coefficient $c_0$. The corresponding boundary condition is therefore given by \eqref{FBCIntrooo}. 

Such results highlight the following facts.
\begin{itemize}
\item[i)] Equivalence between boundary conditions: the generator of a drifted Brownian motion appears to be intimately connected with the (time) tempered fractional derivative;
\item[ii)] Equivalence between multiplicative functionals: the elastic drifted Brownian motion and the tempered subordinator are intimately related. In particular, some functionals of $\widetilde{X}^{\pm\mu}$ can be written in terms of $L_t$ with tempering parameter $\eta=(\pm \mu /2)^2$ and $T_\mu$ with $\mu\geq 0$. 
\end{itemize}

\subsection{Motivations and discussion of the results}
Our aim is to underline the connection between the non-local dynamic boundary value problem with the well-known Cauchy problem involving the drifted Brownian motion. The alternative formulation of the problem is therefore given in terms of the conditions \eqref{FBCIntro} - \eqref{FBCIntrooo}.  Such a formulation relies on the fact that the multiplicative functional charactering the semigroup can be also described by an operator in time. Actually, we obtain that time-dependent (or dynamical) boundary conditions characterize uniquely such a class of functionals. The equivalence between the boundary conditions \eqref{genFBC} and \eqref{BCIntro} for the Cauchy problem \eqref{CP} gives a deep connection between drifted Brownian motions and tempered subordinators. Thus, the interesting connections between the processes $X^{\pm \mu}_t$, $L_t$ and $L_t \wedge T_\mu$ turn out to be evidently useful in applications, simulation and numerical methods. Moreover, our non-local dynamic problem can be regarded as the starting model for a very general motion in higher dimensions. Roughly speaking, a possible reading of the dynamic boundary value problem on a domain $\Omega \cup \partial \Omega$ can be given by considering two evolution equations respectively for the bulk $\Omega$ and the surface $\partial \Omega$. Such evolution equations can be associated with a motion on $\Omega$ and a motion on $\partial \Omega$. Thus, non-local dynamic boundary value problems should be related with non-homogeneous surfaces and the motion on such surfaces turns out to be affected by some anomalies. The results in the present work give some key ideas on this direction by dealing with the simplified case $\partial \Omega=\{0\}$. Recent results concerning dynamical boundary value problem with the Caputo-Dzherbashian derivative have been given in \cite{FBVPdov, FBVP2dov} where a further application has been considered. In particular, non-local operators in the boundary conditions introduce new models for motions on irregular domains. The irregularity of the domain is due to the boundary in which the process may spend an infinite (mean) amount of time. The present work has been inspired by \cite{SW1,SW2} where the authors have obtained a beautiful characterization of the sticky Brownian motion in terms of a time-dependent boundary condition. We have been also moved by the fundamental awareness that fractional powers of operators (and therefore non-local operators) are strictly related with their local higher-order counterparts, when they exist (as discussed in \cite{DovSPL} and many other interesting papers). For example, the intuitive representation of $(-\Delta)^{1/2}u$ can be given by $(\partial_t)^{1/2}u$ if $\partial_t u = \Delta u$. For the non-local case we are dealing with an object like $\Phi(\partial_t)$. The case $\Phi(\lambda)= \lambda$ corresponds to the ordinary derivative, in this case the dynamical boundary condition have a clear physical interpretation (see \cite{Gold2006}).

\section{Non-local operators and random times}
\label{secDerSub}

We introduce the processes $L_t$ with $\lambda$-potentials
\begin{align}
\label{propPotL}
\mathbf{E}_0 \left[ \int_0^\infty e^{-\lambda t} \, f(L_t)\, dt \right] = \frac{\Phi(\lambda)}{\lambda} \int_0^\infty  e^{-s\, \Phi(\lambda)}\, f(s)\, ds
\end{align}
where the symbol $\Phi$ is a Bernstein function uniquely characterized by the measure $\Pi$ as follows
\begin{align}
\label{symbGenPhi}
\Phi(\lambda) = \int_0^\infty ( 1 - e^{- s \lambda} ) \, \Pi(ds). 
\end{align}
It is well-known that the process $L_t$ can be regarded as the inverse to a subordinator with symbol $\Phi$, say $H_t$, for which we have
\begin{align}
\label{symbH}
\mathbf{E}[e^{-\lambda H_t}] = e^{-t \Phi(\lambda)}, \quad \lambda>0, \quad t>0
\end{align}
and the relation $\mathbf{P}_0(L_t < s) = \mathbf{P}_0(H_s > t)$ holds true. In this context, the measure $\Pi$ is termed L\'{e}vy measure of $H_t$. Both processes are random times in the sense that they are non-negative and non-decreasing. The subordinator $H_t$ may have jumps, thus the inverse $L_t$ defined as
\begin{align}
\label{relationInverse}
L_t := \inf\{s > 0\,:\, H_s > t\}, \quad t > 0
\end{align}
is a continuous process with non-decreasing paths. We also assume that $H_0=0$ and $L_0=0$. We denote by $\ell$ and $h$ the density of $L_t$ and $H_t$ respectively, that is
\begin{align*}
\mathbf{P}_0(L_t \in ds) = \ell(t,s)\, ds
\quad \textrm{and} \quad
\mathbf{P}_0(H_t \in ds) = h(t,s)\, ds.
\end{align*}
As usual, $\mathbf{P}_x$ denotes the probability measure of the process started at $x$. We notice that, by definition of inverse process, $L_t$ is the first exit time of $H_t$ from the interval $(0,t)$. Since $H_t$  has strictly increasing paths with jumps (we are not including the case $\Pi((0, \infty))<\infty$, the Poisson case for instance) the process $L_t$ has continuous paths with plateaus. This is an interesting aspects introducing the concept of delayed and rushed motions for time-changed processes (\cite{CapDov}). \\

We provide the following result which will be useful further on.

\begin{proposition}
\label{proppotentialL}
Let $\theta>0$ be fixed. Let $\Phi$ be the symbol defined in \eqref{symbGenPhi}. Then, for $x \geq 0$,
\begin{align}
\label{LT1}
\int_0^\infty e^{-\lambda t} \mathbf{E}_0 \left[\frac{1- e^{-\theta (L_t -x)}}{\theta} \, \mathbf{1}_{(L_t\geq x)} \right] dt = \frac{1}{\lambda} \frac{1}{\theta + \Phi(\lambda)} e^{-x\, \Phi(\lambda)}, \quad \lambda >0
\end{align}
and
\begin{align}
\label{LT2}
\int_0^\infty e^{-\lambda t} \mathbf{E}_0[e^{-\theta (L_t -x)} \mathbf{1}_{(L_t \geq x)}]\, dt = \frac{\Phi(\lambda)}{\lambda} \frac{1}{\theta + \Phi(\lambda)} e^{-x\, \Phi(\lambda)}, \quad \lambda > 0
\end{align}
hold true.
Moreover
\begin{equation}\label{eq:ltmu-pot}
	\mathbb E \left[ \int_0^\infty e^{-\lambda t} \, f(L_t \wedge T_\mu)\, dt \right] = \frac{\Phi(\lambda) + \mu}{\lambda} \tilde f(\Phi(\lambda)+\mu)
\end{equation}
where $ \tilde f(\lambda) = \int_{0}^{\infty} e^{-\lambda s} f(s) \,\mathrm d s. $

\end{proposition}
\begin{proof}
 First we notice that
\begin{align*}
\int_0^t h(s,x) dx = \mathbf{P}_0(H_s \leq t) = \mathbf{P}_0(L_t \geq s) = \int_s^\infty \ell(t, x)dx, \qquad t>0, \; s>0.
\end{align*}
Formula \eqref{LT1} can be obtained by considering the Tonelli's theorem and the fact that 
\begin{align*}
\frac{1}{\lambda} \frac{1}{\theta + \Phi(\lambda)} e^{-x\, \Phi(\lambda)} 
= & \frac{1}{\lambda} \int_0^\infty e^{- w (\theta + \Phi(\lambda)) - x \Phi(\lambda)} \, dw\\
= & \left[ \textrm{by \eqref{symbH}} \right]\\
= & \frac{1}{\lambda} \int_0^\infty e^{-w \theta} \mathbf{E}_0[e^{-\lambda H_{w+x}}]\, dw\\ 
= & \int_0^\infty e^{-\lambda t} \int_0^\infty e^{-w \theta}  \left[ \int_0^t h(w+x, s) ds \right]\,  dw\, dt\\
= & \int_0^\infty e^{-\lambda t} \int_0^\infty e^{-w \theta} \,  \mathbf{P}_0(H_{w+x} \leq t)\,  dw\, dt\\
= & \int_0^\infty e^{-\lambda t} \int_0^\infty e^{-w \theta} \,  \mathbf{P}_0(L_t \geq w + x) \, dw\, dt\\
= & \int_0^\infty e^{-\lambda t} \mathbf{E}_0 \left[ \int_0^\infty e^{-w \theta} \,  \mathbf{1}_{(L_t -x \geq w)} \, dw\right] dt\\
= & \int_0^\infty e^{-\lambda t} \mathbf{E}_0 \left[ \left( \int_0^{L_t -x} e^{-w \theta} dw\right) \mathbf{1}_{(L_t\geq x)}\right] dt\\
= & \int_0^\infty e^{-\lambda t} \mathbf{E}_0 \left[\frac{1- e^{-\theta (L^\Phi_{t}-x)}}{\theta} \, \mathbf{1}_{(L_t\geq x)} \right] dt, \quad \lambda>0.
\end{align*}

Formula \eqref{LT2} immediately follows from \eqref{propPotL}.

We write $ \Phi = \Phi(\lambda) $ for short. By applying \eqref{propPotL} we get
\begin{align}
	&\mathbf E \left[ \int_0^\infty e^{-\lambda t} \, f(L_t \wedge T_\mu)  \, dt \right ]
	\\
	&=
	\int_0^\infty e^{-\lambda t} 
	\left[
	\mathbf E\left( f(T_\mu) \mathbf 1(L_t > T_\mu) \right) 
	+ \mathbf E\left( f(L_t) \mathbf 1(L_t \leq T_\mu) \right) 
	\right]
	\notag \\
	&=
	\mathbf E \left[ f(T_\mu) \mathbf E \left[ \int_{0}^{\infty} e^{-\lambda t} \mathbf 1_{[T_\mu, \infty)} (L_t) \mathrm d t \Big| T_\mu \right] \right] + 
	\mathbf E \left[ \mathbf E \left[ \int_{0}^{\infty} e^{-\lambda t} f(L_t) \mathbf 1_{[0, T_\mu]} (L_t) \mathrm d t \Big| T_\mu \right] \right] \notag \\
	&=
	\frac{\Phi}{\lambda} \mathbf E \left[
	f(T_\mu)  \int_0^\infty e^{-\Phi s} \mathbf 1_{[T_\mu, \infty)} (s) \mathrm d s + \int_0^\infty e^{-\Phi s} \mathbf 1_{[0, T_\mu]} (s) f(s) \mathrm d s
	\right]
	\notag \\
	&=
	\frac{\Phi}{\lambda} \mathbf E \left[
	f(T_\mu) \frac{e^{-\Phi T_\mu}}{\Phi} + \int_0^{T_\mu} e^{-\Phi s} f(s) \mathrm d s
	\right]
	\notag \\
	&=
	\frac{1}{\lambda}
	\int_0^\infty \mu e^{-(\mu + \Phi )z}f(z ) \, \mathrm d z + 
	\frac{\Phi}{\lambda}
	\int_0^\infty \mu e^{-\mu z}
	\int_0^{z} e^{-\Phi s} f(s) \mathrm d s
	\notag \\
	&=
	\frac \mu \lambda \tilde f ( \mu + \Phi )  + 
	\frac{\Phi}{\lambda}
	\int_0^\infty e^{-(\mu + \Phi )s}f(s ) \, \mathrm d s 
	\notag \\
	&=
	\frac {\mu + \Phi}\lambda \tilde f ( \mu + \Phi ) .
	\notag 
\end{align}

\end{proof}

The non-local operator associated with $H_t$ is given by (Bochner-Phillips)
\begin{align}
\label{PhiRL}
-\Phi(-\partial_x) \psi(x) := \int_0^\infty \left( \psi(x) - \psi(x-s)  \right) \Pi(ds), \quad x\geq 0.
\end{align}
Indeed, from \eqref{symbGenPhi}, the Laplace transform of the right-hand side of \eqref{PhiRL} gives  
\begin{align*}
\left( \int_0^\infty (1- e^{-\lambda s}) \, \Pi(ds) \right) \widetilde{\psi}(\lambda)= \Phi(\lambda)\, \widetilde{\psi}(\lambda)
\end{align*} 
for a function $\psi$ compactly supported on the positive real line. Thus, in the Laplace analysis, the symbol $\Phi$ turns out to be the multiplier of the operator \eqref{PhiRL}. Formula \eqref{PhiRL} recall the definition of fractional derivative given by Marchaud, thus we may refer to \eqref{PhiRL} as a Marchaud (type) operator (the definition coincides in case of stable subordinator, that is for $\Phi(\lambda)=\lambda^\alpha$). An interesting discussion about the comparison between fractional derivatives has been given in \cite{Ferr}. The Riemann-Liouville (type) operator is therefore written for a general symbol $\Phi$ as
\begin{align*}
\mathcal{D}^{\Phi}_x \psi(x) := \frac{d}{dx} \int_0^x \psi(x-s)\, \overline{\Pi}(s)\,ds
\end{align*}
where $\overline{\Pi}(s) = \Pi((s, \infty))$ is the tail of $\Pi$. We can check that the symbol $\Phi$ still plays the role of multiplier for this operator, that is
\begin{align}
\label{multiplierRL}
\int_0^\infty e^{-\lambda x} \mathcal{D}^{\Phi}_x \psi(x)\, dx = \Phi(\lambda)\, \widetilde{\psi}(\lambda).
\end{align}

We now introduce the time fractional operator we will deal with further on. Let $N>0$ and $n\geq 0$. Let $\mathcal{N}_\omega$ be the set of (piecewise) continuous function on $[0, \infty)$ of exponential order $\omega$ such that $|\psi(t)| \leq N e^{\omega t}$. Denote by $\widetilde{\psi}$ the Laplace transform of $\psi$. Then, we define the operator $\mathfrak{D}^\Phi_t : \mathcal{N}_\omega \mapsto \mathcal{N}_\omega$ as the Caputo (type) operator for which
\begin{align}
\label{symDerPhi}
\int_0^\infty e^{-\lambda t} \mathfrak{D}^{\Phi}_t \psi(t)\, dt = \Phi(\lambda)\, \widetilde{\psi}(\lambda) - \frac{\Phi(\lambda)}{\lambda} \psi(0), \quad \lambda > \omega.
\end{align}
This immediately introduces the definition
\begin{align}
\label{relRL-C-T}
\mathfrak{D}^{\Phi}_t \psi(t) := \mathcal{D}^{\Phi}_t \psi(t) - \Pi((t, \infty)) \psi(0) 
= \mathcal{D}^{\Phi}_t \left( \psi(t) - \psi(0) \right)
\end{align}
where we have used formula \eqref{multiplierRL} and the well-known fact (\cite[Section 1.2]{Bertoin})
\begin{align}
\label{LapTail}
\int_0^\infty e^{-\lambda t} \, \Pi((t, \infty))\, dt = \frac{\Phi(\lambda)}{\lambda}.
\end{align}
The identity  $\mathcal{D}^{\Phi}_t \psi(0)= \Pi((t, \infty)) \psi(0)$ follows from the definition of $\mathcal{D}^{\Phi}_t$. Since $\psi$ is exponentially bounded, the integral $\widetilde{\psi}$ is absolutely convergent for $\lambda>\omega$. Since $\Phi(\lambda) \widetilde{\psi}(\lambda) - \Phi(\lambda)/\lambda \, \psi(0) = \left( \lambda \widetilde{\psi}(\lambda) - \psi(0) \right)\Phi(\lambda)/\lambda $, then $\mathfrak{D}^\Phi_t$ can be written as a convolution involving the ordinary derivative $\psi^\prime$ and the tail $\Pi((t, \infty))$ iff $\psi \in \mathcal{N}_\omega \cap C((0, \infty), \mathbb{R}_+)$ and $\psi^\prime \in \mathcal{N}_\omega$. In particular, 
\begin{align}
\label{CDder}
\mathfrak{D}^{\Phi}_t \psi(t) = \int_0^t \psi^\prime(t-s)\, \overline{\Pi}(s)\, ds.
\end{align}
By Young's inequality for convolution and formula \eqref{LapTail} we have that
\begin{align}
\label{YoungSymb}
\int_0^\infty |\mathfrak{D}^\Phi_t \psi |^p dt \leq \left( \int_0^\infty |\psi^\prime |^p dt \right) \left( \lim_{\lambda \downarrow 0} \frac{\Phi(\lambda)}{\lambda} \right)^p, \qquad p \in [1, \infty)
\end{align}
where 
\begin{align}
\label{ratioLimPhi}
\lim_{\lambda \downarrow 0} \frac{\Phi(\lambda)}{\lambda} = \frac{d \Phi}{d\lambda}(\lambda) \bigg|_{\lambda=0}
\end{align}
is finite only in some cases. The limit \eqref{ratioLimPhi} will be considered again further on and it is related with the mean value of the subordinator $H_t$. Indeed, from \eqref{symbH},
\begin{align*}
\mathbf{E}_0[H_t] = t \, \frac{d \Phi}{d\lambda}(\lambda) \bigg|_{\lambda=0}.
\end{align*}
We notice that when $\Phi(\lambda)=\lambda$ (that is we deal with the ordinary derivative $D_t$) the equality holds true \eqref{YoungSymb} and $H_t = t$, $L_t=t$ almost surely. 
Some further representations of $\mathfrak{D}^\Phi_t$ in terms of the tails of a L\'evy measure $\Pi((t,\infty))$ have been given in the recent works \cite{Chen, Toaldo} and previously in \cite{Koc2011}. \\

Assuming that
\begin{align}
\label{Assumption}
\lim_{\lambda \downarrow 0} \frac{\Phi(\lambda)}{\lambda} < \infty,
\end{align}
formula \eqref{YoungSymb} says that we are looking for $\psi \in C((0, \infty))$ with $\psi^\prime \in L^1((0, \infty))$. Thus, the minimal requirement is that $\psi \in AC((0,\infty))$. As usual, we denote by $AC((0,\infty))$ the set of absolutely continuous functions on $(0, \infty)$. In particular, $\psi \in AC((0, \infty))$ if $\psi \in C((0, \infty))$ and $\psi^{\prime} = \varrho \in L^1((0, \infty))$, that is we can write
\begin{align}
\label{generalizedDer}
\psi(t) = \psi(0) + \int_0^t \varrho(s)ds.
\end{align}

Let us denote by $C_b((0, \infty))$ the set of smooth and bounded functions on $(0, \infty)$. In order to give a clear picture about the operator \eqref{relRL-C-T}, under the assumption \eqref{Assumption}, we now address the problem to find $\rho(t,x)$ such that $\rho\in C^{1,1}((0, \infty), (0,\infty); (0, \infty))$ and $\forall\, x>0$, $\rho(\cdot, x) \in AC((0, \infty))$ solving
\begin{equation}
\label{probell}
\left\lbrace
\begin{array}{ll}
\displaystyle \mathfrak{D}^{\Phi}_t \rho(t,x) = - \frac{\partial \rho}{\partial x}(t,x), \quad t>0,\; x>0,\\
\displaystyle \rho(0,x)= f(x) \in C_b([0, \infty)),\\
\displaystyle \rho(t,0)= 0, \quad t>0.
\end{array}
\right.
\end{equation}
Then, there is a (classical) solution 
\begin{align}
\label{minimalReq}
\rho \in C^{1,1}(AC((0, \infty)), (0, \infty); (0, \infty))
\end{align}
with probabilistic representation
\begin{align*}
\rho(t,x) = \mathbf{E}_0[f(x - L_t) \mathbf{1}_{(t < H_x)}]
\end{align*}
where $L_t$ is an inverse to a subordinator $H_t$ with symbol $\Phi$. We can easily verify such results. Let us denote by $\widehat{\rho}(t,\xi)= \int_0^\infty e^{-\xi x} \rho(t,x)\, dx$ and $\widetilde{\rho}(\lambda, x) = \int_0^\infty e^{-\lambda t} \rho(t,x)\, dt$ the Laplace transforms w.t. to the variables $x$ and $t$ respectively. Let $\widehat{\widetilde{\rho}}(\lambda, \xi)$ be the double Laplace transform. With \eqref{symDerPhi} at hand, from the problem \eqref{probell} we write  
\begin{align*}
\Phi(\lambda) \, \widehat{\widetilde{\rho}}(\lambda, \xi) - \frac{\Phi(\lambda)}{\lambda} \widehat{f}(\xi) = -\xi \widehat{\widetilde{\rho}}(\lambda, \xi) 
\end{align*}
from which
\begin{align*}
\widehat{\widetilde{\rho}}(\lambda, \xi) = \frac{\Phi(\lambda)}{\lambda} \frac{1}{\xi + \Phi(\lambda)} \widehat{f}(\xi), \quad \lambda>0,\; \xi >0. 
\end{align*}
From Proposition \ref{proppotentialL} we get that
\begin{align*}
\rho(t,x) = \int_0^x f(y)\, \ell(t,x-y)\, dy = \mathbf{E}_0[f(x - L_t) \mathbf{1}_{(L_t < x)}].
\end{align*}
The Laplace machinery gives uniqueness.  The probabilistic representation follows by considering \eqref{relationInverse}. As we can see $\forall\, x>0$, $\rho(\cdot, x) \in L^1((0, \infty))$ only under \eqref{Assumption}. This agrees with \eqref{YoungSymb}. If the strong assumption \eqref{Assumption} does not hold, then we have to ask for 
\begin{align*}
\varrho^\prime(t-s) \overline{\Pi}(s) \in L^1((0, t)), \quad \forall \, t>0. 
\end{align*}

Despite the minimal requirement \eqref{minimalReq} we notice that $\ell(\cdot, x) \in C^{\infty}((0, \infty))$ for any $x>0$. It suffices to consider, for a given $x>0$, the function
\begin{align*}
R_n(\lambda) = \lambda^n \int_0^\infty e^{-\lambda t} \ell(t,x)\,dt = \lambda^n \frac{\Phi(\lambda)}{\lambda} e^{-x \Phi(\lambda)}, \quad \lambda>0, \quad n \in \mathbb{N}_0.
\end{align*}
Since $\Phi$ is a Bernstein function with $\Phi(0)=0$, we get that
\begin{align*}
\lim_{\lambda \to 0} R_n(\lambda) = 0, \quad \lim_{\lambda \to \infty} R_n(\lambda) =0, \quad \forall\, n\in \mathbb{N}.
\end{align*}
This also prove that $\ell(\cdot, x) \notin L_1((0, \infty))$ for any $x>0$ except in case \eqref{YoungSymb} is in force.

Furthermore, we only notice that the kernel $\ell$ can be uniquely determined as the solution to the problem
\begin{equation}
\label{probellDirac}
\left\lbrace
\begin{array}{ll}
\displaystyle \mathcal{D}^{\Phi}_t \ell(t,x) = - \frac{\partial \ell}{\partial x}(t,x), \quad t>0,\; x>0,\\
\displaystyle \ell(0,x)= \delta(x) \\
\displaystyle \ell(t,0)= \Pi((t, \infty)),
\end{array}
\right.
\end{equation}
where $\delta$ is the Dirac function and the derivative \eqref{PhiRL} is considered in place of \eqref{relRL-C-T}. The Laplace technique can by applied as before by considering the formula \eqref{LapTail}. The problem \eqref{probellDirac} has been investigated in \cite{Toaldo}. In the literature very often this equations are confused in the sense that, only the first one can be written in terms of the Caputo type derivative. Sometimes the boundary condition is omitted. Below we are interested in a kind of fractional relaxation equation  based on \eqref{probell}.

\section{Tempered fractional calculus}
\label{sec:temp}

From now on we focus on the symbol
\begin{align}
\label{symbGenSub}
\Phi(\lambda) = \sqrt{\lambda+\eta} - \sqrt{\eta}, \quad \lambda \geq 0
\end{align}
corresponding to the L\'{e}vy measure
\begin{align}
\label{levMeas}
\Pi(ds) = \frac{1}{2}\frac{1}{\sqrt{\pi}} \frac{e^{-\eta s}}{s^{\frac{1}{2} + 1}} ds, \quad \eta >0.
\end{align}
We recall that the corresponding subordinator $H_t$ is the \textit{tempered}  (also termed \textit{relativistic}) stable subordinator of order $ \frac 12 $. The measure of a tempered stable processes can be obtained by multiplying the Lévy measure of an $ \alpha $-stable process by a decreasing exponential. The parameter $ \eta > 0 $ controls the level of tempering. The effect is to reduce the intensity of large jumps keeping the structure of small jumps. The resulting process has finite moments of all order and at the same time, it has an infinite amount of (small) jumps in any finite time interval. For these reasons these models are widely studied, see e.g. \cite{CGMY} for applications in mathematical finance  or 
or \cite{MeeSabChe} and references therein for applications to hydrology problems.
Anomalous diffusion with tempered operators were considered in \cite{Cartea-DCN}, while a general theory for tempering stable processes was presented in \cite{Ros}. 

\autoref{fig:sub} compares the sample paths of a stable subordinator and of a tempered stable subordinator,
showing that the presence of the tempering parameter reduces the number of larger jumps.

\begin{figure}
	\centering
	\begin{subfigure}{0.475\textwidth}
		\centering
		\begin{overpic}[width=.8\textwidth]{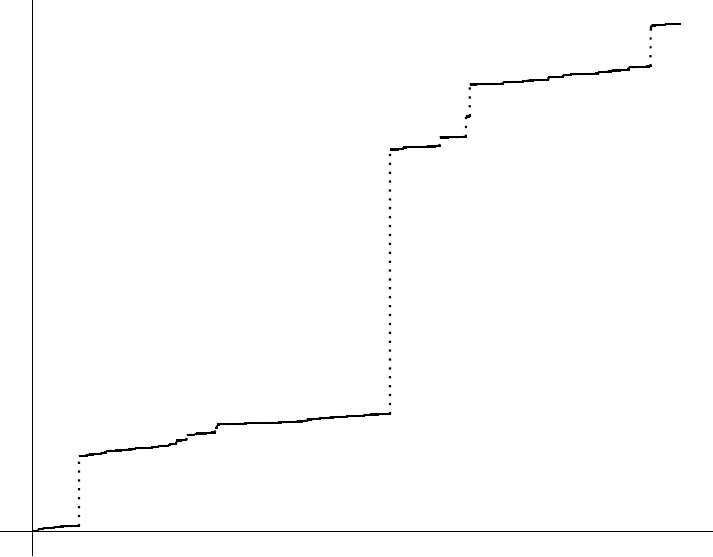}
		\end{overpic}
		\caption{sample path of $ H_t , \eta =0$}
	\end{subfigure}
	\hfill
	\begin{subfigure}{0.475\textwidth}
		\centering
		\begin{overpic}[width=.8\textwidth]{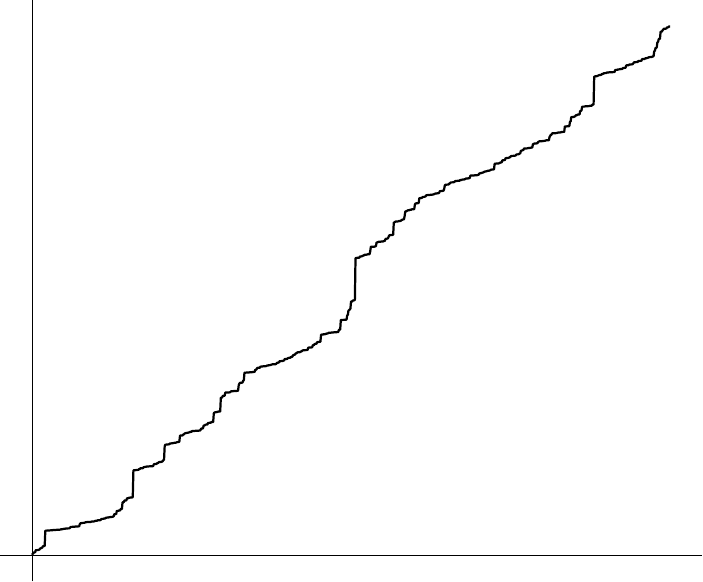}
		\end{overpic}
		\caption{sample path of $ H_t , \eta >0$}
	\end{subfigure}	
	\caption{Comparison of a sample path of a stable subordinator ($ \eta=0 $) and a tempered stable subordinator ($ \eta >0$). Both pictures must be interpreted with jumps in place of vertical lines. The paths are not continuous.}
	\label{fig:sub}
\end{figure}

The Caputo (type) tempered derivative is given by 
\begin{align}
\label{defDerTemp}
\mathfrak{D}^{\frac{1}{2},\eta}_t \psi(t) = \int_0^t \psi^\prime(s) \overline{\Pi}(t-s) ds
\end{align}
where $\overline{\Pi}(z)=\Pi((z, \infty))$ is the tail of the L\'evy measure $\Pi$ given in \eqref{levMeas}. From \eqref{YoungSymb},  we obtain that
\begin{align}
\label{YoungSymbTemp}
\big\| \mathfrak{D}^{\frac{1}{2},\eta}_t \psi \big\|_{L^1} \leq \frac{1}{\sqrt{2\eta}} \, \| \psi^\prime \|_{L^1}
\end{align}
which may be of interest only if $\eta \neq 0$. It is well known that, for $\eta=0$, 
$$\ell(t,x) = 2 e^{-\frac{x^2}{4t}} / \sqrt{4\pi t}, \quad t>0,\, x>0.$$  
The symbol \eqref{symbGenSub} for $\eta=0$ introduces the following derivatives: 
\begin{itemize}
	\item the Riemann-Liouville derivative
	\begin{align*}
	\mathcal{D}^\frac{1}{2}_t \psi(t) = \frac{1}{\sqrt{\pi}} \frac{d}{dt} \int_0^t \frac{\psi(s)}{\sqrt{t-s}} ds 
	\end{align*}
	\item the Caputo-Djrbashian derivative
	\begin{align*}
	\mathfrak D^\frac{1}{2}_t \psi(t) = \frac{1}{\sqrt{\pi}} \int_0^t \frac{\psi^\prime(s)}{\sqrt{t-s}} ds
	\end{align*}
	where $\psi^\prime = d\psi/ds$.
\end{itemize}

We recall the $\lambda$-potential
\begin{align}
\label{propPotL0}
\mathbf{E}_0 \left[ \int_0^\infty e^{-\lambda t} \, f(L_t)\, dt \right] = \frac{\sqrt{\lambda + \eta} - \sqrt{\eta}}{\lambda} \int_0^\infty  e^{-s (\sqrt{\lambda + \eta} - \sqrt{\eta})}\, f(s)\, ds
\end{align}
(which can be obtained as special case of the formula \eqref{propPotL} by considering the L\'{e}vy measure
\eqref{levMeas}) and the following formula
\begin{align}
\label{symbDermu}
\int_0^\infty e^{-\lambda t} \mathfrak D^{\frac{1}{2}, \eta}_t \psi(t) \,dt = (\sqrt{\lambda + \eta} - \sqrt{\eta}) \widetilde{\psi} - \frac{1}{\lambda} (\sqrt{\lambda + \eta} - \sqrt{\eta})\psi_0
\end{align}
which will be useful in the subsequent discussion. The interested reader can consult for example \cite{Beg, DovOrsIaf, MeeSabChe} for further discussions on this operator and tempered processes.

In the following we consider $\eta= \frac{\mu^2}{4}$ as a tempering parameter. Thus in order to streamline the notation as much as possible we write
\begin{align}
\label{DerSakeSimpl}
\mathfrak{D}^{\frac{1}{2},\mu}_t \psi(t) = \mathfrak D^{\frac{1}{2}, \eta}_t \psi(t), \quad \textrm{with }\; \eta= \frac{\mu^2}{4}, \; \mu>0
\end{align}


\begin{proposition}
\label{propSolBCgeneral}
Let $a,b$ be two positive constants. The unique continuous solution on the interval $I \subseteq [0, \infty)$ of the fractional tempered relaxation equation

\begin{equation}
\left\lbrace
\begin{array}{ll}
\displaystyle \mathfrak{D}^{\frac{1}{2}, \mu}_t r(t) + a\, r(t) = b, & t > 0\\
\displaystyle r(0)= c\in \{0,1\} & 
\end{array}
\right.
\end{equation}
is given by
\begin{align*}
r(t) = \frac{b}{a}\, \mathbf{P}_0(L_t \geq T_a) + c\,  \mathbf{P}_0 (L_t < T_a)
\end{align*}
where $T_a$ is an exponential random variable (with parameter $a$) independent from $L_t$ which is the inverse process with symbol \eqref{symbGenSub}.
 
If $c=1$, then $[0, \infty) =I \ni t \mapsto r(t)$ has the following properties:
\begin{itemize}
\item[i)] $b>a \; \Rightarrow \; r(t)>1$; 
\item[ii)] $b<a \; \Rightarrow \; r(t)<1$;
\item[iii)] $b=a \; \Rightarrow \; r(t)=1$.
\end{itemize}

If $c=0$, then $r(t)$ has the following properties:
\begin{itemize}
\item[iv)] $b>a>0$  $\Rightarrow \, \exists \, t_b : r(t) \leq 1 \; \textrm{ if }\; t \in I=[0, t_b)$;
\item[v)] $b\leq a \, \Rightarrow \, r(t) \leq 1, \; t \in I=[0, \infty)$.
\end{itemize}

Moreover, $\forall\, a,b,c$ the solution $t \mapsto r(t)$ is monotone with $r(0)=c$ and $r(t) \to b/a$ as $t \to \infty$.
\end{proposition}

\begin{proof}
From \eqref{symbDermu}, by Laplace techniques we obtain
\begin{align*}
\widetilde{r}(\lambda) 
= & \int_0^\infty e^{-\lambda t} r(t)\, dt, \quad \lambda \geq 0\\
= & \frac{1}{\lambda} \frac{b + c \sqrt{\lambda+\eta} - \sqrt{\eta}}{a + \sqrt{\lambda+\eta} - \sqrt{\eta}} \\
= & \frac{b}{\lambda} \frac{1}{a + \sqrt{\lambda+\eta} - \sqrt{\eta}} + c \frac{\sqrt{\lambda+\eta} - \sqrt{\eta}}{\lambda} \frac{1}{a + \sqrt{\lambda+\eta} - \sqrt{\eta}}.
\end{align*}
We recall that $\Phi(\lambda)=\sqrt{\lambda+\eta} - \sqrt{\eta}$ is a completely monotone function for which $\Phi(0)=0$ and $\Phi(\lambda)\to \infty$ as $\lambda\to \infty$. We immediately see that $r(t)$ is a non-negative solution. 

From Proposition \ref{proppotentialL} with $\Phi(\lambda)= \sqrt{\lambda+\eta} - \sqrt{\eta}$ we write
\begin{align*}
r(t) 
= & \frac{b}{a}\,  \mathbf{E}_0[1-e^{-a L_t}] + c\, \mathbf{E}_0[e^{-aL_t}]\\
= & \frac{b}{a} \, \mathbf{P}_0(L_t \geq T_a) + c\, \mathbf{P}_0 (L_t < T_a) \\
= & c + \left( \frac{b}{a} - c \right) \mathbf{P}_0(L_t \geq T_a) .
\end{align*}

Let us consider $c=1$. For $a=b$, if follows that $r(t)=\mathbf{P}_0(T_a \in [0, \infty))=1$ $\forall\, t> 0$. Moreover, there exist $\varepsilon_1=\varepsilon_1(t) \geq 0$ and $\varepsilon_2=\varepsilon_2(t) \in [0,1)$ such that, $\forall\, t>0$
\begin{equation}
r(t) = \left\lbrace
\begin{array}{ll}
\displaystyle 1 + \varepsilon_1, & \textrm{if}\; b>a\\
\displaystyle 1 - \varepsilon_2, & \textrm{if}\; b<a
\end{array}
\right . .
\end{equation}
Since $\mathbf{P}_0(0=L_0 \leq T_a)=0$ we recover the initial condition $r(0)=1$.

Now we focus on $c=0$. 
\begin{itemize}
\item[($b<a$):] Since $r(t)=(b/a) \mathbf{P}_0(L_t \geq T_a)$, $r(t) \leq 1$ follows immediately for $b < a$. Indeed,
\begin{align*}
\widetilde{r}(\lambda) \leq \frac{1}{\lambda} \frac{a}{a + \sqrt{\lambda+\eta} - \sqrt{\eta}} \leq \frac{1}{\lambda};
\end{align*}
\item[($b=a$):] If $b=a$ and $a\to 0$, then $r(t)\to 0$ $\forall\, t \geq 0$. If $b=a$ and $a\to \infty$, then $r(t)\to 1$ $\forall\, t>0$. We simply have $\widetilde{r}(\lambda) \leq 1/\lambda$ from which $r(t)\leq 1$ for any $t$ with $b=a \geq 0$;
\item[($b>a$):] Let $b< \infty$. We have that, $r(t) \to 1- \mathbf{E}_0[e^{-b L_t}] \leq 1$ uniformly in $[0, \infty)$ as $a\to b$. The crucial point is given by the fact that $r(t) \to b \mathbf{E}_0[L_t]$ pointwise in $[0, \infty)$ as $a\to 0$ (recall that $L_0=0$). In this case, $r(t)\leq 1$ iff $\mathbf{E}_0[L_t] \leq 1/b$ with $b>0$. Let us denote by $L_t^0$ the process $L_t$ with $\eta=0$. Since $\lambda^\beta$ is $\alpha$-H\"older continuous on $[0, \infty)$ only for $\beta=\alpha$ we obtain that $\sqrt{\lambda+\eta}-\sqrt{\lambda} \leq \sqrt{\eta}$ which implies $\sqrt{\lambda+\eta} - \sqrt{\eta} \leq \sqrt{\lambda}$. The comparison between symbols and the fact that
\begin{align*}
\int_0^\infty e^{-\lambda t} \mathbf{E}_0[L_t]\, dt = \frac{1}{\lambda} \frac{1}{(\sqrt{\lambda + \eta} - \sqrt{\eta})} \geq \frac{1}{\lambda} \frac{1}{\sqrt{\lambda}} = \int_0^\infty e^{-\lambda t} \mathbf{E}_0[L_t^0]\, dt
\end{align*} 
says that $\mathbf{E}_0[L_t] \geq \mathbf{E}_0[L_t^0]$, $t\geq 0$. From the fact that
\begin{align*}
\mathbf{E}_0[e^{-L_t^0}] = \sum_k \frac{(-1)^k}{k!} \mathbf{E}_0[(L_t^0)^k] \quad \textrm{equals} \quad \sum_k \frac{(-\sqrt{t})^k}{\Gamma(k/2 + 1)} = E_\frac{1}{2}(-\sqrt{t})
\end{align*}
we get the well-known result
\begin{align*}
\mathbf{E}_0[L_t^0] = \frac{\sqrt{t}}{\Gamma(1/2+1)}
\end{align*}
which implies $\mathbf{E}_0[L_t^0] \geq 1$ for $t\geq \pi/4$. Recall tht $L_0^0= L_0 = 0$. Thus, $r(t)\leq 1$ only in some bounded domain $[0, t_b) \subset [0, \infty)$. 
\end{itemize}
The monotonicity of $r(t)$ follows by considering that
\begin{align*}
r(t) = c + C(a,b,c)\, \mathbf{P}_0(T_a \leq L_t) = c + C(a,b,c)\, \mathbf{P}_0(H_{T_a} \leq t)
\end{align*}
where we have used the relation \eqref{relationInverse}. Since  $\mathbf{P}_0(H_{T_a} \leq t)$ is a cumulative distribution function, the result follows.
\end{proof}

Recently fractional relaxation equations has been considered in \cite{BegGaj}. The authors have obtained similar results for $b=0$ and $\mu <1$ involving the gamma random variable $\mathfrak{G}$ with density $\mathbf{P}(\mathfrak{G} \in ds)= s^{\mu -1} / \Gamma(\mu)\, e^{-s} ds$, that is $r(t) = \mathbf{P}_0(\mathfrak{G} > a^{1/\mu}\, t)$. Interesting discussions have been made in the papers \cite{Baz2018, Koc2011} and the pioneering work \cite{Mainardi}. In \cite{Baz2018, Koc2011} the properties of the solutions to fractional relaxation equations in terms of complete monotone functions have been investigated.

\begin{remark}
For $c=1$ the solution $r(t)$ is monotone increasing or decreasing depending on the ratio $b/a$, that is respectively $b/a>1$ or $b/a<1$. For $c=0$ the solution $r(t)$ is only increasing.
\end{remark}

\begin{remark}
The initial datum $c\in \{0, 1\}$ will be related with the fact that, for a given Markov process and the corresponding multiplicative functional $M_t$ we have $M_0 \in \{0, 1\}$. Indeed, form the relation $M_{t+s}=M_t(M_s\circ M_t)$ we obtain $M_0=M_0^2$ which implies that almost surely $M_0$ is either $0$ or $1$.
\end{remark}

\section{Elastic drifted Brownian motions}
\label{sec:EDBM}

We introduce and study here the elastic drifted Brownian motion, for short we often write EDBM. We also write RBM meaning a reflecting Brownian motion.
Let us consider the process $\widetilde{X}^\mu= \{\widetilde{X}^\mu_t\}_{t\geq 0}$ on $[0, \infty)$ with generator  $(G_\mu, D(G_\mu))$ where
\begin{align*}
	G_\mu \varphi = \mu \frac{d \varphi}{dx} + \frac{d^2 \varphi}{d x^2}
\end{align*}
and
\begin{align*}
	D(G_\mu) = \left\lbrace \varphi, G_\mu \varphi \in C_b((0, \infty))\,:\, \varphi^\prime (0^+) = c \, \varphi(0^+) \right\rbrace.
\end{align*}
The constant $c>0$ is termed \emph{elastic coefficient}. 
%
%
The transition density of an elastic Brownian motion with drift is 
given by
\begin{align}
	\label{density:pmu}
	&p(t,x,y) 
	\\
	&=  e^{-\frac{\mu^2}{4}t} e^{\frac{\mu}{2}(y-x)}
	\, 
	\left[ g(t, x-y) + g(t, x+y) - 2 \left(c + \frac \mu 2 \right) \int_0^\infty e^{\left(c + \frac \mu 2 \right) \,  w} g(t, w+x+y) dw
	\right] \notag 
\end{align}
for $ x\geq 0,\; y>0, t>0 $, where $g(t,z)=e^{-z^2/4t}/\sqrt{4\pi t}$ and $c \geq 0$. See the Appendix for some hints 
on the derivation of \eqref{density:pmu}.
%
%
In \cite{io-sojourn} the authors highlight an interesting connection between the law of drifted elastic Brownian motions \eqref{density:pmu} and conditional sojourn times of a Brownian motion on the positive half-axis.  
The solution to the Cauchy problem
\begin{align*}
	\partial_t u = G_\mu u, \quad u_0=f \in D(G_\mu)
\end{align*}
is written as
\begin{align*}
	u(t,x)= \int_0^\infty f(y) p(t,x,y)dy=\mathbf{E}_x[f(\widetilde{X}^\mu_t)]
\end{align*}
and the semigroup generated by $(G_\mu, D(G_\mu))$ has the probabilistic representation
\begin{align}
	\label{solmuSemig}
	P^\mu_t f(x) = \mathbf{E}_x[f(\hat X^\mu_t) M_t^\mu] = \mathbf{E}_x[f(\widetilde{X}^\mu_t)]
\end{align}
where $\hat {X}^\mu_t$ is a drifted Brownian motion on $[0, \infty)$ reflected at $0$ and $M^\mu_t$ is the multiplicative functional associated with the Robin boundary condition. 
Let 
\[
G_\lambda(x, y) = \int_{0}^{\infty} e^{-\lambda t} p(t, x, y) \, \mathrm dt, 
\qquad x, y > 0	
\]
be the Green function and
\[
R_\lambda f(x) = \int_{0}^{\infty} e^{-\lambda t} P^\mu_t f(x) \, \mathrm d t
= \int_{0}^{\infty} G_\lambda(x, y) f(y) \,\mathrm d y
\]
be the resolvent associated to the EDBM. Detailed expressions are provided in the Appendix.

\begin{remark}
	\label{remarkLife}
	We recall some basic facts which will be useful in the forthcoming discussion.
	
	Let $G_0=\Delta$ be the infinitesimal generator for some Brownian motion on $E$. The probabilistic representation of the solution to
	\begin{align*}
		& \frac{\partial w}{\partial t} = Gw, \quad w_0 = \mathbf{1}
	\end{align*}
	with some boundary conditions can be written as $w(t,x)= \mathbf{E}_x[M_t] = \mathbf{E}_x[e^{-A_t}]$ that is, in terms of the multiplicative functional $M_t$ or equivalently in terms of the corresponding additive functional $A_t$. For the Robin boundary condition $(\partial_{\bf n} w + c \,w)|_{\partial E}=0$, we have that $M_t= \mathbf{1}_{(t < \zeta)}$ where $\zeta$ is the lifetime of the process with generator $(G, D(G))$. The additive functional to be considered is the local time $\gamma_t$. In particular,
	\begin{align}
		\label{rmkLife1}
		w(t,x) =  \mathbf{E}_x [e^{-c \gamma_t}] = \int_0^\infty e^{-cw} \mathbf{P}_x(\gamma_t \in dw) = 1 - \int_0^\infty ( 1 - e^{-cw} ) \,\mathbf{P}_x(\gamma_t \in dw)
	\end{align}
	or equivalently
	\begin{align}
		\label{rmkLife2}
		w(t, x) =  \mathbf{E}_x [\mathbf{1}_{(t<\zeta)}] = \int_t^\infty \mathbf{P}_x(\zeta \in ds)= 1- \int_0^t \mathbf{P}_x(\zeta \in ds).
	\end{align} 
	
	It is well known that $\mathbf{P}_x(\zeta > t) = \mathbf{P}_x (T_c > \gamma_t)$ where $T_c$ is an exponential random variable with parameter $c>0$ independent from the local time $\gamma_t$ on $\partial E$. The connection between \eqref{rmkLife1} and \eqref{rmkLife2} immediately emerges.
	
	It is well-known that $\gamma_t$ equals in law the running maximum of a Brownian motion started at $x=0$. Moreover, such an equivalence in distribution is maintained with the inverse to an $1/2$-stable subordinator. Notice that, such an inverse process corresponds to $L_t$ with $\eta=(\mu/2)^2=0$.
\end{remark}

Our first results are concerned with the relation between the the inverse to a tempered stable subordinator and  the drifted (reflecting) Brownian motion together with its maximum and its local time. These relations will be useful in the following in connection with the multiplicative functional associated to the EDBM. Results will differ 
if the underlying Brownian motion has a positive or negative drift. We study the two cases separately.\\

\textbf{Remark on the notation.} \emph{For the reader's convenience, in the following discussion, we only allow $ \mu > 0 $, so that a positive drift will be denoted by $ \mu $ and a
negative drift by $ -\mu $.}

\subsection{Brownian motion with positive drift, RBM with negative drift}

In this section we study the case where the Brownian motion $ X^\mu $ has positive drift $ \mu >0 $.

\begin{theorem}
\label{thm0}
For the positively drifted Brownian motion $X^\mu$ with $X^\mu_0=0$, we have that
\begin{align}
\label{eq:equal-xmu-inv}
\max_{0\leq s \leq t} X^\mu_s \stackrel{d}{=} L_t, \quad t>0
\end{align}
where $L$ is an inverse to a relativistic stable subordinator with symbol \eqref{symbGenSub}, $\eta=(\mu/2)^2$.
\end{theorem}
\begin{proof}
Formula \eqref{eq:equal-xmu-inv} can be shown by a Laplace transform argument. The distribution of the maximum of a Brownian motion with drift $ \mu $ is well-known. To the best or our knowledge, the law of the maximum has been obtained in \cite{Shep79, CFS87} together with the joint law with its location. For our purposes we refer to \cite{Iafrate19} (with some adaptation) and write
\begin{align}\label{eq:bm-drift}
&\mathbf P_x\left(
\max_{0 \leq s \leq t} X^{\mu}_s > \beta
\right) 
= 
\int_{\beta}^{\infty} 
\frac{e^{- \frac{(z-x)^2}{4t}}}{\sqrt{4 \pi t}} e^{- \frac{\mu^2t}{4} - \frac \mu  2 x} 
\left[
e^{\frac \mu 2 z} + e^{\frac \mu 2 (2 \beta - z)} 
\right] \mathrm dz \,\,,\qquad  \beta > x.
\end{align}
A direct computation immediately shows that the Laplace transform of \eqref{eq:bm-drift} is 

\begin{align}\label{eq:lap-max-bmu}
	\int_{0}^{\infty} 
	e^{-\lambda t} \mathbf P_0\left(
	\max_{0 \leq s \leq t} X^{\mu}_s > \beta
	\right) dt
	&=
	\int_{\beta}^{\infty}	
	\left[
	e^{\frac \mu 2 z} + e^{\frac \mu 2 (2 \beta - z)} 
	\right] \int_{0}^\infty e^{-\left( \frac{\mu^2}{4} + \lambda \right)t}
	\frac{e^{- \frac{z^2}{4t}}}{\sqrt{4 \pi t}} \, \mathrm d t \, \mathrm d z 
	\\
	&=
	\int_{\beta}^{\infty}	
	\left[
	e^{\frac \mu 2 z} + e^{\frac \mu 2 (2 \beta - z)} 
	\right]
	\frac{e^{-z \sqrt{\lambda + \frac{\mu^2} 4}}}{2\sqrt{\lambda + \frac{\mu^2} 4}}
	\mathrm d z 
	\notag \\
	&=
	\frac{e^{-\beta \left(\sqrt{\lambda + \frac{\mu^2} 4} - \frac \mu 2\right )  }}{2\sqrt{\lambda + \frac{\mu^2} 4}}
	\left[
	\frac{1}{\sqrt{\lambda + \frac{\mu^2} 4} + \frac \mu 2} +
	\frac{1}{\sqrt{\lambda + \frac{\mu^2} 4} - \frac \mu 2} 
	\right] 
	\notag \\
	&=
	\frac 1 \lambda e^{-\beta \left(\sqrt{\lambda + \frac{\mu^2} 4} - \frac \mu 2 \right)  }
	\notag 
\end{align}
where we used the well-known formula \eqref{eq:lap-gauss} recalled in the Appendix.

%

On the other hand, by letting $ \theta \to 0  $ in \eqref{LT2}, we immediately see that for the inverse tempered subordinator 
with symbol \eqref{symbGenSub}, $ \eta = \frac{\mu^2}{4} $, it holds that 

\begin{equation}\label{eq:lap-inv-temp}
\int_{0}^{\infty} e^{-\lambda t} \mathbf P_0(L_t > \beta) \,\mathrm d t
=
\frac 1 \lambda e^{-\beta \left(\sqrt{\lambda + \frac{\mu^2}{4}} - \frac{\mu}{2} \right)  }
\end{equation}
thus proving the equality in distribution \eqref{eq:equal-xmu-inv}.

\end{proof}

Moreover we point out a further interesting connection between the tempered subordinator and the local time of the drifted Brownian motion. First we introduce the process $\{ Y^{\theta, \sigma}_t\}_{t \geq 0}$ as the unique strong solution to 
\begin{equation}\label{eq:sde-pre-reflecting}
\mathrm d Y^{\theta, \sigma}_t = -\theta \, \mathrm{sgn} Y^{\theta, \sigma}_t + \sigma\mathrm dB_t\,, \qquad Y^{\theta, \sigma}_0 = 0
\end{equation}
 where $B_t$ is a standard Brownian motion, $\theta \in \mathbb{R}$ and $ \sigma >0 $. In the following we will restrict ourselves to the cases $ \theta = \pm \mu/2,\, \sigma = \sqrt 2 $  and for brevity we define $ Y^\mu_t \coloneqq Y^{\mu/ 2, \sqrt 2}_t, t\geq 0$.
Denote with $ \{ \gamma_t(Y^\mu) \}_{t \geq 0} $ the local time process of $ Y^\mu = \{Y^\mu_t\}_{t \geq 0}$. Analogously we define $ Y^{-\mu}_t \coloneqq Y^{-\mu/ 2, \sqrt 2}_t, t\geq 0$ and  $ \{ \gamma_t(Y^{-\mu}) \}_{t \geq 0} $ as the corresponding local time.

\begin{corollary}
\label{cor:inverse-local-time}
For the local time (at zero) of $Y^\mu$ we have that
\begin{align}
\label{relationLgammaY}
\gamma_t(Y^\mu) \stackrel{d}{=} L_t, \quad t>0.
\end{align}
\end{corollary}
\begin{proof}
In \cite[Theorem 1]{GravShir} the authors prove the equality in distribution
\begin{align}
\label{resultGravShir}
\left(\max_{0\leq s \leq t} X^\mu_s - X^\mu_t, \max_{0\leq s \leq t} X^\mu_s \right) \stackrel{d}{=} \big(|Y^\mu_t|, \gamma_t({Y^\mu)}\big)
\end{align}

The result follows from \eqref{eq:equal-xmu-inv} and \eqref{resultGravShir}.
\end{proof}

In \cite{GravShir} the authors show that $ |Y^\mu| $ 
constitutes a reflecting Brownian motion with drift $ -\mu $.
For $\mu=0$, the relation \eqref{relationLgammaY} agrees with the well-known equality in distribution between maximum, local time and inverse to a $1/2$-stable subordinator as described in Remark \ref{remarkLife}. However, when the presence of the drift is assumed, a fundamental difference emerges. For $\mu > 0$, that is for $\eta>0$, the inverse tempered subordinator is related to the local time of the process $ Y^\mu $ instead of the local time of a Brownian motion with drift.

\subsection{Brownian motion with negative drift, RBM with positive drift}
We now consider the case where the underlying Brownian motion $X^{-\mu}$ has negative drift $ -\mu < 0 $.

The result of Corollary \autoref{cor:inverse-local-time} relates the distribution of the inverse of a tempered subordinator with the distribution of $ Y^\mu $, which is in turn related to a reflecting Brownian motion (RBM for short) with \emph{negative} drift. If one starts with a RBM with positive drift, i.e. by considering the process $ Y^{-\mu} = \{Y^{-\mu}_t\}_{t\geq 0}$ and its absolute value, the symmetry appears to break. 
In fact, while the equality in distribution \eqref{resultGravShir} still holds, relating the RBM with positive drift $ |Y^{-\mu}| $ and the local time  $ \gamma_t (Y^{-\mu}) $ with a Brownian motion with negative drift $ X^{-\mu} $ and its maximum, these processes are not directly related anymore to the inverse of a tempered 
stable subordinator. It is instead necessary to introduce a ``truncated version" of the inverse subordinator
as in the following theorem.

\begin{theorem}
\label{thm00}
For the negatively drifted Brownian motion $X^{-\mu}$ with $X^{-\mu}_0=0$, we have that
\begin{align}
\label{eq:equal-xmu-inv-neg}
\max_{0\leq s \leq t} X^{-\mu}_s \stackrel{d}{=} L_t \wedge T_\mu, \quad t>0
\end{align}
where $T_\mu$ is an exponential r.v. (with parameter $\mu>0$) independent from $L$ which is an inverse to a relativistic stable subordinator with symbol \eqref{symbGenSub}, $\eta=(-\mu/2)^2$.
\end{theorem}
\begin{proof}

We check that the Laplace transforms of the distribution of both sides of \eqref{eq:equal-xmu-inv-neg} coincide. 
The Laplace transform 
of the distribution of the maximum \eqref{eq:lap-max-bmu} in this case becomes

\begin{align}\label{eq:lap-max-bmu-neg}
\int_{0}^{\infty} 
e^{-\lambda t} \mathbf P_0\left(
\max_{0 \leq s \leq t} X^{-\mu}_s > \beta
\right) dt
&=
\frac 1 \lambda e^{-\beta \left(\sqrt{\lambda + \frac{\mu^2} 4} + \frac \mu 2 \right)  }.
\end{align}

Note that \eqref{eq:lap-max-bmu-neg} is now different from \eqref{eq:lap-inv-temp}, where the tempering parameter cannot be negative. This is why \autoref{thm0} does not apply in this case. 

Now, by considering \eqref{eq:lap-max-bmu-neg} and \eqref{eq:lap-inv-temp} we have that 
\begin{align}\label{eq:lap-max-bmu-neg-vs-inv}
\int_{0}^{\infty} 
e^{-\lambda t} \mathbf P_0\left(
\max_{0 \leq s \leq t} X^{-\mu}_s > \beta
\right) dt
&=
\frac 1 \lambda e^{-\beta \left(\sqrt{\lambda + \frac{\mu^2} 4} + \frac \mu 2 \right)  } 
=
e^{- \mu \beta}
\frac 1 \lambda e^{-\beta \left(\sqrt{\lambda + \frac{\mu^2} 4} - \frac \mu 2 \right)  } 
\\
&=
\int_{0}^{\infty} e^{-\lambda t} e^{-\mu \beta} \mathbf P_0(L_t > \beta) \,\mathrm d t
\notag \\
&=
\int_{0}^{\infty} e^{-\lambda t} g^\mu(\beta, t) \,\mathrm d t
 \,\,,\qquad \beta > 0.
 \notag 
\end{align}
where $g^\mu(\beta, t) =  e^{-\mu \beta}P(L_t > \beta) $. The quantity 
$ 1 - g^\mu(\beta, t) $ coincides with the distribution function of $ L_t \wedge T_\mu $, where $ T_\mu $ is an independent exponential random variable with parameter $ \mu $ and $ L_t $ is  assumed to start from zero. 
In fact, by independence,
\begin{align}
\label{eq:lt-min-proof}
	\mathbf P_0( L_t \wedge T_\mu > \beta) &=
	P(L_t>\beta) \, \mathbf E \left(\mathbf 1_{(T_\mu > \beta)} \right)
\end{align}
Thus by \eqref{eq:lap-max-bmu-neg} the result is proved. 
\end{proof}

\begin{corollary}
\label{cor:inverse-local-time-negative}
For the local time (at zero) of $Y^{-\mu}$ we have that
\begin{equation}
\label{relationLgammaY-negtivative} 
\gamma_t( Y^{-\mu})  \stackrel{d}{=} L_t \wedge T_\mu \, \qquad  t >0
\end{equation}
where $T_\mu$ is an independent exponential r.v. with parameter $\mu$.
\end{corollary}
\begin{proof}
By applying the same arguments as in the proof of Corollary \eqref{cor:inverse-local-time} we can show that \eqref{relationLgammaY-negtivative} holds true. We stress the fact that $ Y^{-\mu} $ is the process such that $ |Y^{-\mu}| $ is a RBM with \emph{positive drift} $ \mu $.
\end{proof}

Let us discuss the Figures we enclose to our presentation. \autoref{fig:brown} shows some sample paths of the processes $ Y^{\mu} $ and $ Y^{-\mu}  $ as well 
as the corresponding reflecting processes $  |Y^{\mu}| $ and  $ |Y^{-\mu}|  $. We see that in the case of the RBM positive drift, i.e. $ |Y^{-\mu}|  $,  the sample paths tend to travel further from the barrier, 
whereas in the case of negative drift, i.e. $ |Y^{\mu}|  $, the sample path is constantly pushed towards the barrier. This gives an intuitive explanation of the difference between the relations \eqref{relationLgammaY} and  \eqref{relationLgammaY-negtivative}. In the second case since the process $ Y^{-\mu}  $ travels away from the barrier its local time at zero tends to stop increasing. This corresponds to the fact that the local time in this case has the same distribution of a randomly truncated inverse subordinator. \autoref{fig:inv-sub} shows a comparison between the sample paths of an inverse stable subordinator and the paths of the maximum of a drifted Brownian motion. In particular \autoref{fig:inv-sub-a} shows two sample paths of $ L_t $ while \autoref{fig:inv-sub-b} shows the same sample paths randomly truncated with exponential random variables (blue horizontal lines), i.e. realizations of 
$ L_t \wedge T_\mu $. \autoref{fig:inv-sub-c} shows a sample path of a Brownian motion with \emph{positive} drift and its running maximum $ \max_{0\leq s \leq t}  X^\mu_s$. The similarity with the sample paths in 
\autoref{fig:inv-sub-a} illustrates the equality in distribution proved in \autoref{thm0}.
\autoref{fig:inv-sub-d} shows a Brownian motion with \emph{negative} drift and its maximum. Note that 
as the sample paths travels away from zero the maximum stops increasing,  exhibiting a behavior 
similar to the paths in \autoref{fig:inv-sub-b}: this is the thesis of \autoref{thm00}.

\begin{figure}[h]
	\centering
	\begin{subfigure}{0.475\textwidth}
		\centering
		\begin{overpic}[width=.8\textwidth]{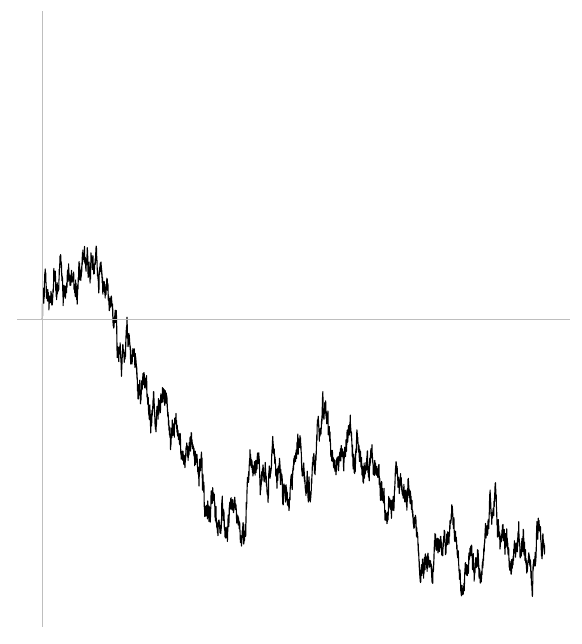}
		\end{overpic}
		\caption{sample path of $ Y^{-\mu}_t $ \\ \hphantom{a}}
	\end{subfigure}
	\hfill
	\begin{subfigure}{0.475\textwidth}
		\centering
		\begin{overpic}[width=.8\textwidth]{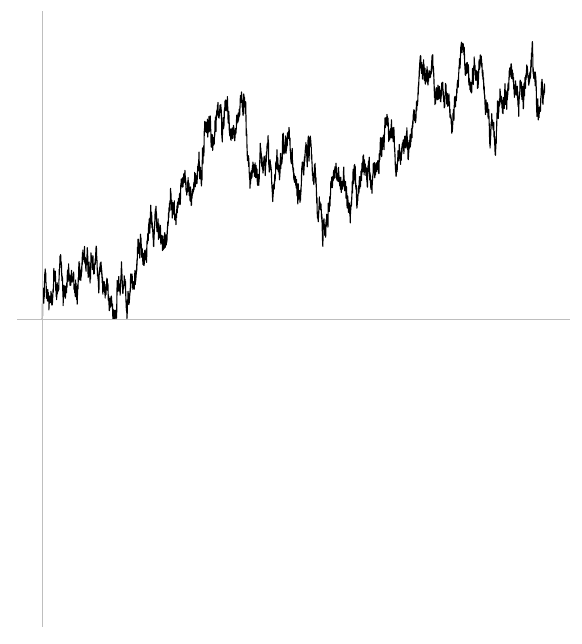}
		\end{overpic}
		\caption{sample path of $ |Y^{-\mu}_t| $, i.e. an RBM with positive drift.}
	\end{subfigure}	
	
	\begin{subfigure}{0.475\textwidth}
		\centering
		\begin{overpic}[width=.8\textwidth]{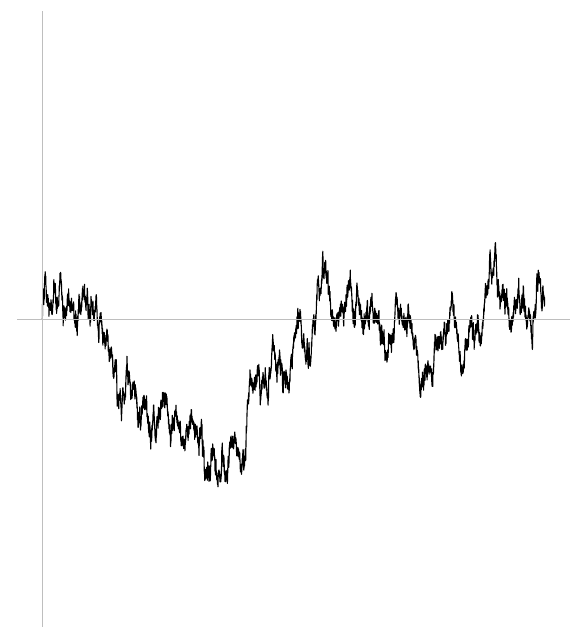}
		\end{overpic}
		\caption{sample path of $ Y^{\mu}_t $ \\ \hphantom{a}}
	\end{subfigure}
	\hfill
	\begin{subfigure}{0.475\textwidth}
		\centering
		\begin{overpic}[width=.8\textwidth]{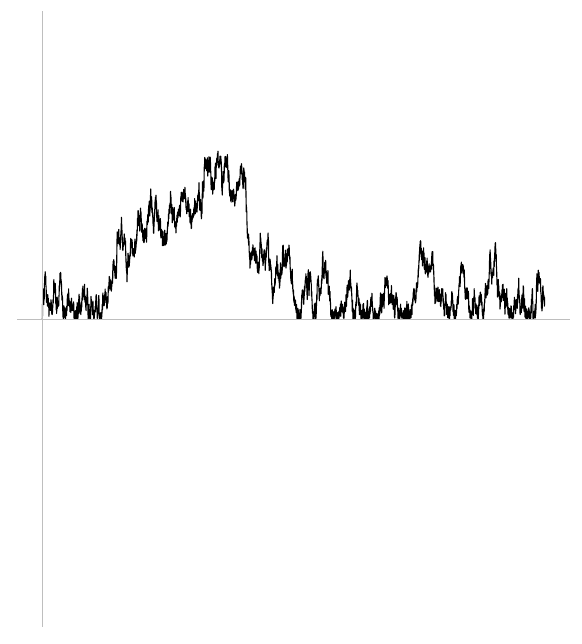}
		\end{overpic}
		\caption{sample path of $ |Y^{\mu}_t| $, i.e. an RBM with negative drift.}
	\end{subfigure}	
	
	\caption{Comparison of a sample paths of the solution to \eqref{eq:sde-pre-reflecting} and the corresponding RBM with drift.}
	\label{fig:brown}	
	
\end{figure}

\begin{figure}[h]
	\centering
	\begin{subfigure}{0.475\textwidth}
		\centering
		\begin{overpic}[width=.8\textwidth]{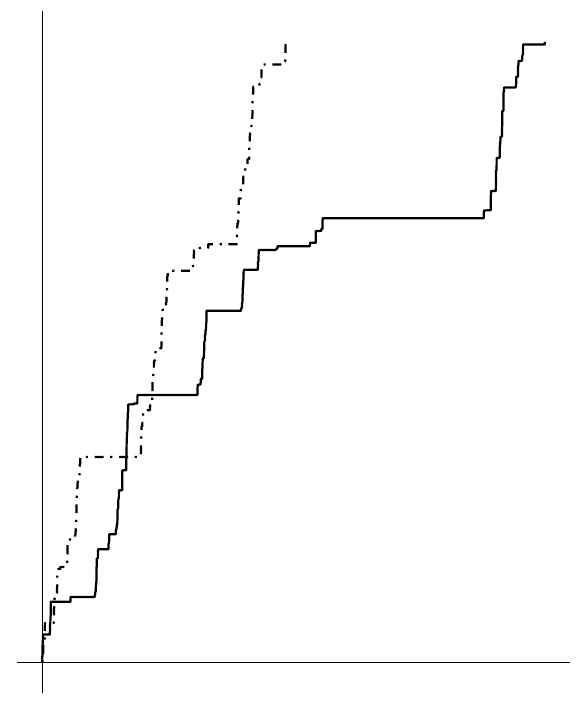}
		\end{overpic}
		\caption{sample paths of $ L_t $}
		\label{fig:inv-sub-a}
	\end{subfigure}
	\hfill
	\begin{subfigure}{0.475\textwidth}
		\centering
		\begin{overpic}[width=.8\textwidth]{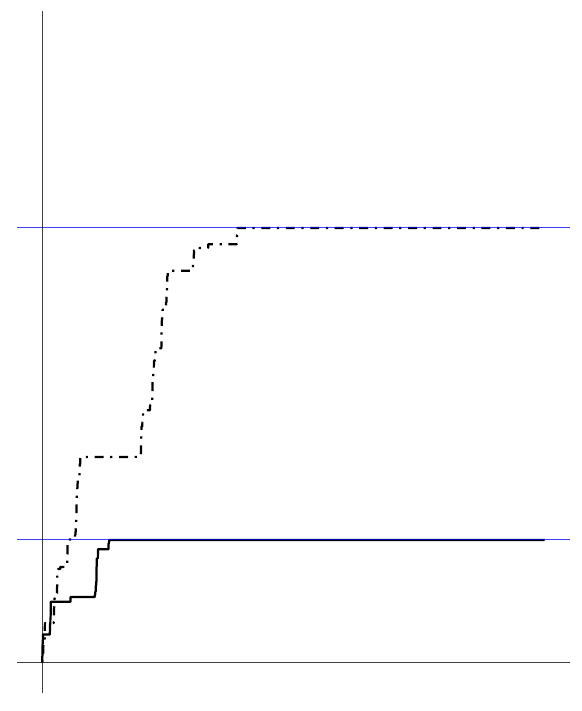}
		\end{overpic}
		\caption{sample paths of $ L_t \wedge T_\mu $}
		\label{fig:inv-sub-b}
	\end{subfigure}	
	
	\begin{subfigure}{0.475\textwidth}
		\centering
		\begin{overpic}[width=.8\textwidth]{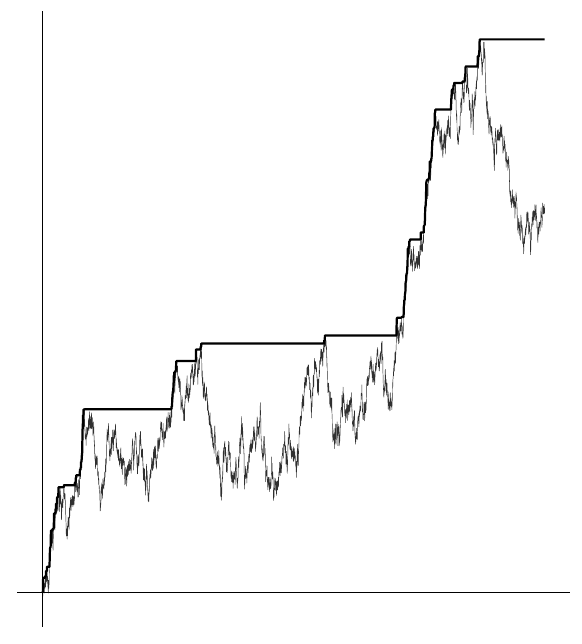}
		\end{overpic}
		\caption{sample path of $ X_t^\mu $ and its running maximum $ \max_{0\leq s\leq t} X^\mu_s $}
		\label{fig:inv-sub-c}
	\end{subfigure}
	\hfill
	\begin{subfigure}{0.475\textwidth}
		\centering
		\begin{overpic}[width=.8\textwidth]{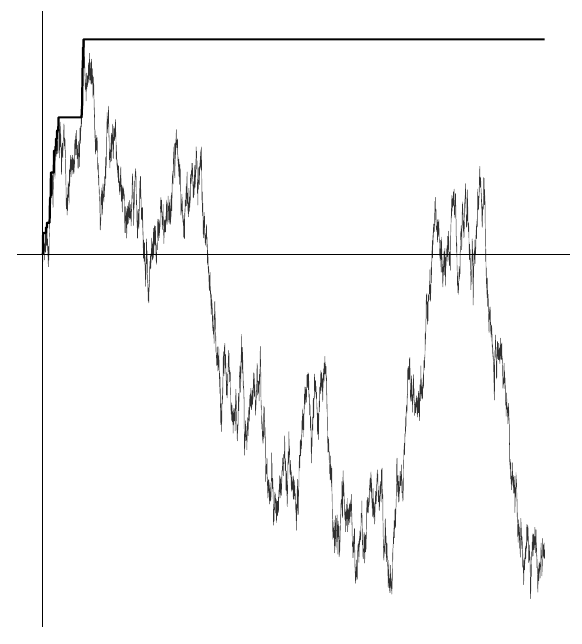}
		\end{overpic}
		\caption{sample path of $ X_t^{-\mu }$ and its running maximum $ \max_{0\leq s\leq t} X^{-\mu}_s $ }
		\label{fig:inv-sub-d}
	\end{subfigure}	
	
	\caption{Comparison of a sample path of an inverse tempered stable subordinator and the maximum of a Brownian motion with drift}
	\label{fig:inv-sub}
	
\end{figure}

\section{Helpful intuitive introduction to fractional boundary conditions}
\label{helpfulSection}
Here we consider a particular and instructive case which gives an helpful and intuitive interpretation of the main result of Section \ref{sec:caseI}. The proofs of the following statements have been postponed in the Appendix.

Let us consider the generator $(G_0, D(G_0))$ of the reflected Brownian motion on $[0, \infty)$ with elastic condition at $x=0$ for which we have that
\begin{align}
\label{helpfulSemig}
P^0_t \mathbf{1}(x) =  \int_0^\infty \big( g(t,x-y) - g(t,x+y)\big) \, dy + 2 \int_0^\infty e^{-c \, w} g(t,w+x) \, dw.
\end{align}
First we observe that formula \eqref{helpfulSemig} has the following representation 
\begin{align}
P^0_t \mathbf{1}(x) =  1 - \int_0^t \frac{x}{s} g(s,x) \, ds + 2 \int_0^\infty e^{-c \, w} g(t,w+x) \, dw=: F(t,x) \label{helpfulProof}
\end{align}
in which the density of the lifetime $\zeta$ emerges as mentioned in Remark \ref{remarkLife}.  Formula \eqref{helpfulProof}, in turn, can be written by considering the non-negative and non-decreasing process $A_t$ as
\begin{align}
\label{helpfulProbRep}
F(t,x) = & 1 - \mathbf{P}_0(A^{-1}_x \leq t) + e^{x \,c} \mathbf{E}_0[ e^{-c\, A_t} \mathbf{1}_{(A_t>x)}], \quad t\geq 0, \; x \geq 0
\end{align}
for which, at the boundary point $x=0$, we get  
\begin{align*}
F(t,0) = \mathbf{E}_0[e^{-c\, A_t}], \quad t\geq 0. 
\end{align*}
The process $A^{-1}_t = \inf\{s\geq 0\,:\, A_s \geq t\}$ is the inverse to $A_t$. We have the following interesting cases at the boundary point $x=0$:
\begin{itemize}
\item[i)] 
$A_t = \gamma_t$ is the Brownian local time and the usual condition writes
\begin{align}
\label{BChelpful1}
\frac{\partial F}{\partial x}(t,0) = c\, F(t,0).
\end{align}
The elastic condition  \eqref{BChelpful1} introduces exponential solutions.
\item[ii)]
$A_t= L_t^0$ is the inverse to a stable subordinator (of order $\alpha=1/2$, we use the superscript and write $L^0$ instead of $L$ to underline that $\eta=0$) and
\begin{align}
\label{BChelpful2}
D^\frac{1}{2}_t F(t,0) = -c\, F(t,0)
\end{align}
whose solution is the Mittag-Leffler function
$$F(t,0)=E_\frac{1}{2}(-c\, \sqrt{t}).$$
The elastic condition \eqref{BChelpful2} introduces solutions to relaxation equations.
\item[iii)] The boundary condition 
\begin{align}
\label{GovEqAHelpful}
D^\frac{1}{2}_t F(t,0) = - \frac{\partial F}{\partial x}(t,0)
\end{align}
holds true. Despite the fact that we lose the dependence from the elastic coefficient $c_0$, we get information about the additive functional. Indeed, \eqref{GovEqAHelpful} is the governing equation of $L^0$. 
\end{itemize}
Obviously $\gamma_t \stackrel{law}{=} L^0_t$ and their sample paths are both positive and non decreasing with $\gamma_0=L^0_0=0$. Both conditions \eqref{BChelpful1} and \eqref{BChelpful2} give unique characterization of the boundary behaviour of the reflected Brownian motion at $x=0$.

%

\section{Fractional boundary conditions}

\label{sec:caseI}
We discuss here the connection between the infinitesimal generators $G_\mu$ and $G_{-\mu}$ and the tempered derivative of order $1/2$. The order $1/2$ seems to be naturally related to the fact that $G_{\pm \mu}$ is a second order operator (see for example \cite[]{DovSPL}). The drift $\pm \mu$ is related to the tempering parameter $\eta$ of the tempered derivative by means of the relation $\eta = (\pm \mu/2)^2$.

\subsection{The positively drifted Brownian motion}
We focus on the function 
$$u \in C^{1,2}(AC(0, \infty)\times [0, \infty), [0, \infty))$$ solving 
\begin{equation}
\label{probExit}
\begin{cases}
\displaystyle  \frac{\partial u}{\partial t} = \mu \frac{\partial u}{\partial x} + \frac{\partial^2 u}{\partial x^2}\\
\displaystyle u(0, x) = \mathbf{1}(x \geq 0)
\end{cases}
\end{equation}
with the boundary condition
\begin{align}
\label{BCprobprobExit}
\mathfrak{D}^{\frac{1}{2},\mu}_t u(t, 0) + (c + \mu) u(t,0) = \mu, \quad t > 0.
\end{align}
Notice that we consider here the boundary condition \eqref{BCprobprobExit} in place of
\begin{align}
\label{BCstandard}
\frac{\partial u}{\partial x}(t,0) = c\, u(t,0), \quad t>0.
\end{align}
We observe that the condition \eqref{BCprobprobExit} can be rewritten as
\begin{align*}
\int_0^t (G_\mu u)(t-s,0)\,\Pi((s, \infty))\, ds + \mu\, u(t,0) + c \, u(t, 0) = \mu, \quad t>0 
\end{align*}
by following the definition \eqref{defDerTemp}. For $u \in D(G_\mu)$, we formally have
\begin{align*}
\mathfrak{D}^{\frac{1}{2},\mu}_t u(t, 0) + \frac{\mu + c}{c} \frac{\partial u}{\partial x}(t, 0) = \mu, \quad t>0
\end{align*}
or equivalently
\begin{align*}
\int_0^t (G_\mu u)(t-s,0)\,\Pi((s, \infty))\, ds + \frac{\mu + c}{c} \frac{\partial u}{\partial x}(t, 0) = \mu, \quad t>0. 
\end{align*}

Further on we will write $\bar{\mathbf{1}}(x) := \mathbf{1}(x\geq 0)$ in order to streamline the notation.
 
\begin{theorem}
\label{thmI}
Let us consider $u=u(t,x)$ given in \eqref{solmuSemig}. 
Then, the following statements hold:
\begin{itemize}
\item[i)] $u$ is the unique (classical) solution to \eqref{probExit} - \eqref{BCstandard};
\item[ii)] $u$ is the unique (classical) solution to \eqref{probExit} - \eqref{BCprobprobExit};
\item[iii)] $u$ has the probabilistic representation
\begin{align}
\label{solmuSemigRep}
u(t,x) 
= & 1 - \frac{c}{\mu + c} e^{-x\mu} \mathbf{E}_0\left[\left( 1 - e^{-(\mu + c) (L_t - x)} \right) \mathbf{1}_{(L_t \geq x)}\right] 
\end{align}
where $L$ is an inverse to a relativistic stable subordinator with symbol \eqref{symbGenSub}, $\eta=(\mu/2)^2$.
\end{itemize}
\end{theorem}

\begin{remark}
The semigroup $P^\mu_t \bar{\mathbf{1}}(x)$ has the probabilistic representation \eqref{solmuSemigRep}. This means that we can use the properties of $L_t$ in order to obtain equivalence of functionals under $\mathbf{P}_0$ (in mean $\mathbf{E}_0$). Since $P^\mu_t \bar{\mathbf{1}}(x) = \mathbf{E}_x[M^\mu_t]$, the process $L_t$ can be considered in order to obtain information about $M^\mu_t$. We underline that, here, $L_t$ is independent from $(X^\mu_t, M^\mu_t)$, the only advantage we may have is given by the equivalence in  law expressed by \eqref{solmuSemig} and \eqref{solmuSemigRep}. A further interesting reading will be given ahead in Theorem \ref{thmIbis} and Theorem \ref{last-thmI}.  
\end{remark}

\begin{proof}[Proof of Theorem \ref{thmI}]
We proceed step by step, first discussing the point iii). By exploiting the resolvent formula  \eqref{eq:app-res} in the Appendix we have
\begin{align}
\label{LTuTHM1}
\widetilde{u}(\lambda, x) 
= & \int_0^\infty e^{-\lambda t} \, u(t,x)\, dt = R_\lambda \mathbf{ \bar 1} (x)\notag \\
= & \frac{1}{\lambda} - \frac{1}{\lambda} \left( 1 -  \frac{\Phi(\lambda) + \mu}{c + \mu +  \Phi(\lambda)}  \right) e^{- x(\Phi(\lambda)+\mu)} 
\end{align}
where the symbol $\Phi$ denotes
\begin{align*}
\Phi(\lambda) = \sqrt{\lambda+\frac{\mu^2}4} - \frac \mu 2.
\end{align*}
By considering the fact that (see \eqref{propPotL})
\begin{align*}
\int_0^\infty e^{-\lambda t} \mathbf{P}(L_t>x) dt = \int_x^\infty \frac{\Phi(\lambda)}{\lambda} \, e^{-s \Phi(\lambda)}\, ds
\end{align*}
together with Proposition \ref{proppotentialL} for 
\begin{align*}
\frac{\Phi(\lambda)}{\lambda} \frac{1}{c + \mu +  \Phi(\lambda)}  e^{- x(\Phi(\lambda) + \mu)},
\end{align*}
by observing that (recall that $T_a$ is an exponential r.v. with parameter $a$)
\begin{align*}
\frac{\mu}{\lambda} \frac{1}{c + \mu +  \Phi(\lambda)}  e^{- x(\Phi(\lambda) + \mu)} 
= & \frac{\mu}{\lambda} e^{-x\mu} \int_0^\infty e^{-w \, (c + \mu)} e^{- (w+x) \Phi(\lambda)} dw\\
= & \mu e^{-x \mu} \int_0^\infty e^{-\lambda t} \left( \int_0^\infty e^{-w(c + \mu)} \mathbf{P}_0(L_t > x+w)\, dw \right) dt\\
= & \int_0^\infty e^{-\lambda t} \left( \frac{\mu}{\mu + c} e^{-x\mu} \mathbf{P}_0(L_t - x > T_{\mu + c}) \right) dt,
\end{align*}
we obtain the inverse Laplace transform of $\widetilde{u}(\lambda, x)$ given by
\begin{align*}
u(t,x) 
= & 1 - \, e^{-x \mu } \mathbf{P}_0(L_t>x) + e^{-x\mu} \mathbf{E}_0 \left[ e^{- (\mu + c) (L_t - x)} \mathbf{1}_{(L_t \geq x)} \right]\\ 
& + \, \frac{\mu}{\mu + c} e^{-x\mu} \mathbf{E}_0\left[\left( 1 - e^{-(\mu + c ) (L_t - x)} \right) \mathbf{1}_{(L_t \geq x)}\right].
\end{align*}
Simple manipulation leads to \eqref{solmuSemigRep}.

Now we show that \eqref{solmuSemigRep} is the solution to \eqref{probExit} with \eqref{BCprobprobExit}. 

The Laplace transform of the solution as $ x\to 0^+ $ is, by using formula \eqref{eq:el-res0} in Appendix
\begin{align}\label{eq:lap-bd-pos}
	\tilde u (\lambda, 0) &
	= R_\lambda \mathbf{ \bar 1} (0)  =
	\frac 1{\sqrt{\lambda + \frac {\mu^2} 4} + \frac \mu 2 + c} 
	\int_{0}^{\infty} e^{- (\sqrt{\lambda + \frac {\mu^2} 4} - \frac \mu 2 )y} \,\mathrm d y
	\\
	&= 
	\frac{1}{\Phi(\Phi + c + \mu)} \notag 
\end{align} 

By exploiting relation \eqref{eq:lambda-eq}, we see that the Laplace transform of the LHS of \eqref{BCprobprobExit}
is 
\begin{align*}
	\Phi  \tilde u(\lambda, 0) - \frac{\Phi}{\lambda} u(0, 0) + (c+\mu) \tilde u(\lambda, 0)
	&=
	\frac{1}{\Phi + c +\mu } - \frac{\Phi}{\lambda} + 
	\frac{c + \mu}{\Phi(\Phi + c +\mu) }
	\\
	&=
	\frac{1}{\Phi} - \frac{1}{\Phi + \mu } = \frac{\mu}{\Phi(\Phi + \mu)}
	\\
	&=
	\frac{ \mu}{\lambda}.
\end{align*}
which is the Laplace transform of the RHS of \eqref{BCprobprobExit}. 

The point $i)$ does not need to be proved. This concludes the proof.
\end{proof}

\begin{corollary}
\label{coro:equivI}
For the process $\widetilde{X}^\mu_t$, $t\geq 0$ with generator $(G_\mu, D(G_\mu))$ the multiplicative functional  $\overline{M}^\mu_t$ is uniquely characterized by the boundary condition
\begin{align*}
\mathfrak{D}^{\frac{1}{2},\mu}_t u(t, 0) + (c + \mu ) u(t, 0) = \mu, \quad t \geq 0.
\end{align*}
Moreover, $\overline{M}^\mu_t$ is equivalent to $M^\mu_t$.
\end{corollary}
\begin{proof}
Indeed, the potential
\begin{align*}
\mathbf{E}_0\left[ \int_0^\infty e^{-\lambda t} \, M^\mu_t\, dt \right]
\end{align*}
coincides with $\widetilde{u}(\lambda, 0)$ uniquely determined in \eqref{eq:lap-bd-pos} which in turns, uniquely defines the solution to \eqref{probExit} - \eqref{BCprobprobExit} . 
\end{proof}

\begin{remark}
Since $u(0,0)=1$, from Proposition \ref{propSolBCgeneral} we know that $t \mapsto u(t,0)$ is monotone decreasing and such that, for $c>0$,
\begin{align*}
u(t,0) \downarrow \frac{\mu}{\mu + c} \in (0,1) \quad \textrm{as} \quad t\to \infty.
\end{align*}
If $c = 0$, then $u(t,0)=1$ for any $t$. 
\end{remark}

\begin{remark}
	The proof of Theorem \ref{thmI} exploits explicit probability distributions associated to the elastic Brownian motion in order to compute the solution of the  boundary value problem \eqref{probExit}-\eqref{BCstandard} and then check that it also satisfies the fractional condition \eqref{BCprobprobExit}. We note that it is possible to give an alternative proof by directly solving the fractional boundary problem \eqref{probExit}-\eqref{BCprobprobExit}. In fact, take Laplace transforms of \eqref{probExit} to get
	\begin{equation}\label{eq:direct-lap-eq}
	\lambda \tilde u(\lambda, x) - u(0,x) = \mu \partial_x \tilde u(\lambda, x) + \partial^2_{xx} \tilde u(\lambda,x).
	\end{equation}
	with $ u(0,x)=1 $. By taking Laplace transforms of \eqref{BCprobprobExit} we obtain
	\begin{equation}\label{eq:direct-lap-bc}
	\Phi(\lambda) \tilde{u} (\lambda ,0) - \frac{\Phi(\lambda)}{\lambda} u(0,0) + (\mu + c )\tilde u  = \frac \mu \lambda 
	\end{equation}
	with $ \Phi(\lambda) = \sqrt{\lambda + \frac{\mu^2}{4}} - \frac \mu 2 $ and $ u(0,0)=1 $. Thus we obtain
	\begin{equation}\label{key}
	\tilde u (\lambda,0) = \frac{1}{\lambda} \frac{\sqrt{\lambda + \frac{\mu^2}{4}} + \frac \mu 2}{\sqrt{\lambda + \frac{\mu^2}{4}} + \frac \mu 2 + c} =: c_\lambda
	\end{equation}
	Now let $v_\lambda(x) = \tilde{u}(\lambda, x)$. By considering \eqref{eq:direct-lap-eq} and \eqref{eq:direct-lap-bc} we see that \eqref{probExit}-\eqref{BCprobprobExit} may be rewritten as
	\begin{equation}\label{eq:direct-2nd-ode}
	\begin{cases}
	v_\lambda'' + \mu v_\lambda' - \lambda v_\lambda + 1 = 0 \\
	v_\lambda(0) = c_\lambda \\
	v_\lambda (x) \,\mathrm{bounded}
	\end{cases}
	\end{equation}
	which is a second order ODE that can be solved by standard techniques. It is then immediate to check that the unique solution to \eqref{eq:direct-2nd-ode} is precisely \eqref{LTuTHM1}.
\end{remark}

We now recall the condition
\begin{align*}
\mathcal{D}^{\frac{1}{2},\mu}_t u(t,0) + (\mu+c)\, u(t,0) = \mu, \quad t>0
\end{align*}
and the standard condition
\begin{align*}
\frac{\partial u}{\partial x}(t,0) = c \, u(t,0), \quad t>0
\end{align*}
which is associated to \eqref{probExit}. Then, we conclude with the following two results.

\begin{theorem}
\label{thmIbis}
The solution to \eqref{probExit} - \eqref{BCprobprobExit} has the probabilistic representation
\begin{align}
\label{solmuSemigRep2}
\displaystyle u(t,x) = 1 - \mathbf{E}_0 \left[ \left( 1- e^{-c \, (L_t \wedge T_\mu - x)} \right) \mathbf{1}_{(L_t \wedge T_\mu > x)} \right]
\end{align}
where $T_\mu$ is an exponential r.v. (with parameter $\mu>0$) independent from $L$ which is an inverse to a relativistic stable subordinator with symbol \eqref{symbGenSub}, $\eta=(\mu/2)^2$
\end{theorem}
\begin{proof}
Let us write \eqref{LTuTHM1} as follows
\begin{align}
\label{LTuTHM1-new}
\widetilde{u}(t,x) = \frac{1}{\lambda} - \frac{1}{\lambda} \left( 1- (\sqrt{\lambda + \mu^2/4} + \mu/2) \frac{1}{c + \mu/2 + \sqrt{\lambda + \mu^2/4}} \right) e^{-x (\sqrt{\lambda + \mu^2/4} + \mu/2)}.
\end{align}
From Theorem \ref{thm00}, we have that
\begin{align*}
\frac{1}{\lambda} e^{-x (\sqrt{\lambda + \mu^2/4} + \mu/2)} = \int_0^\infty e^{-\lambda t} \mathbf{P}_0(L_t \wedge T_\mu > x)\, dt
\end{align*}
and
\begin{align*}
\frac{\sqrt{\lambda + \mu^2/4} + \mu/2}{\lambda} e^{- x (\sqrt{\lambda + \mu^2/4} + \mu/2)} dx = \int_0^\infty e^{-\lambda t} \mathbf{P}_0(L_t \wedge T_\mu \in dx)\, dt.
\end{align*}
Thus, the integral 
\begin{align*}
\frac{\sqrt{\lambda + \mu^2/4} + \mu/2}{\lambda} \int_0^\infty e^{-w  ( c + (\sqrt{\lambda + (\mu/2)^2} + \mu/2)} e^{-x(\sqrt{\lambda + \mu^2/4} + \mu/2)} dw 
\end{align*}
takes the form
\begin{align*}
\int_0^\infty e^{-\lambda t} \mathbf{E}_0\left[ \int_0^\infty e^{- c (L_t \wedge T_\mu -x)} \mathbf{1}_{(L_t \wedge T_\mu >x)}  \right] dt.
\end{align*}
By collecting all the previous parts, we get that
\begin{align*}
u(t,x) = 1 - \mathbf{E}_0[\mathbf{1}_{(L_t \wedge T_\mu >x)}] + \mathbf{E}_0\left[ \int_0^\infty e^{- c \, (L_t \wedge T_\mu -x)} \mathbf{1}_{(L_t \wedge T_\mu >x)}  \right]
\end{align*}
which is the claimed result.
\end{proof}

\begin{remark}
	Formula \eqref{solmuSemigRep2} can be succinctly represented as 
	\begin{equation}
	\label{solmuSemigRep4}
		u(t,x) = \mathbf P_0(L_t \wedge T_\mu  - x < T_{c})
	\end{equation}
	where $ T_{c} $ is an exponential random variable of parameter $ c $ independent from $ L_t $ and $ T_\mu $, provided that $ c >0  $. This can be justified as follows
	\begin{align*}
	&\mathbf P_0(L_t \wedge T_\mu  - x < T_{c} \cap ((L_t \wedge T_\mu  - x > 0) \cup (L_t \wedge T_\mu  - x < 0)))
	\\
	&=
	\mathbf P_0(L_t \wedge T_\mu  - x < 0) + \mathbf{E}_0 \big[ \,  \mathbf{E} [ \mathbf 1_{ (T_{c} > L_t \wedge T_\mu  - x) } | L_t \wedge T_\mu]  \, \mathbf 1_{ (L_t \wedge T_\mu  - x>0)} \big ]
	\\
	&=
	\mathbf P_0(L_t \wedge T_\mu  < x) + \mathbf{E}_0\left[e^{- c_\mu ( L_t \wedge T_\mu  - x) } \mathbf 1_{ (L_t \wedge T_\mu  > x)}\right].
	\end{align*}
\end{remark}

We now present the last result for the positively drifted Brownian motion.
\begin{theorem}
\label{last-thmI}
The solution to \eqref{probExit} - \eqref{BCprobprobExit} has the probabilistic representation
\begin{align}
\label{solmuSemigRep3}
\displaystyle u(t,x) = 1 - \mathbf{E}_0 \left[ \left( 1- e^{-c \, (S^{-\mu}_t - x)} \right) \mathbf{1}_{(S^{-\mu}_t > x)} \right]
\end{align}
where
\begin{align*}
S^{-\mu}_t = \max_{0\leq s \leq t} X^{-\mu}_s, \quad t>0,\; \mu>0.
\end{align*}
\end{theorem}
\begin{proof}
The proof follows immediately from Theorem \ref{thm00}.
\end{proof}

\begin{remark}
(About the reading of \eqref{BCprobprobExit}) Assume that the boundary conditions
\begin{align*}
\mathfrak{D}^{\frac{1}{2},\mu}_t u(t, 0) + (c + \mu)\, u(t, 0) = \mu 
\end{align*}
and
\begin{align*}
\frac{\partial u}{\partial x}(t,0) = c u(t,0)
\end{align*}
are equivalent. Are we able to obtain information about $M^\mu_t$ from the previous conditions? It seems that the first equation gives immediately the answer we are looking for. If we consider the boundary condition 
\begin{align*}
\mathfrak{D}^{\frac{1}{2},\mu}_t u(t, 0) + c_1\, u(t, 0) = c_2 
\end{align*}
we are able to characterize $M^\mu_t$ in terms of the coefficients $c_1, c_2$ as in the discussion below.

\end{remark}

We analyse some different cases concerned with \eqref{solmuSemigRep} and in particular with the lifetime $\zeta^\mu$ of the process $\widetilde{X}^\mu$. Recall that  
\begin{align*}
P^\mu_t \bar{\mathbf 1}(x) = \mathbf{P}_x(\zeta^\mu > t). 
\end{align*}

\begin{itemize}
	\item[-]
	\emph{Null drift coefficient}. Let us consider $\mu=0.$
	Then, $\forall\, c\geq 0$,
	\begin{align}
	\label{repRemarkDriftCoef}
	P_t^0 \bar{\mathbf{1}}(x) = \mathbf{P}_0(L_t < x) +  \mathbf{E}_0\left[  e^{-c (L_t -x)} \, \mathbf{1}_{(L_t \geq x)} \right], \quad t\geq 0
	\end{align}
	solves the equation \eqref{probExit} with 
	\begin{align}
	\label{solBCdriftProp}
	\mathfrak{D}^{\frac{1}{2},0}_t u(t, 0) + c\, u(t, 0) = 0, \quad t \geq 0
	\end{align}
	where the tempered derivative becomes the Caputo derivative $D^\frac{1}{2}_t$ (see formula \eqref{DerSakeSimpl} with $\eta=0$). Notice that the corresponding multiplicative functional is associated with the elastic Brownian motion with no drift. In particular, $P^0_t \bar{\mathbf{1}}(x) = \mathbf{E}_x[M^0_t]$ where $M^0_t = \exp(-c \gamma_t)$ emerges in case of Robin boundary condition. Let us consider $x=0$ for the sake of simplicity. We immediately see that
	\begin{align*}
	M^0_t \stackrel{law}{=} e^{-c\, L_t}
	\end{align*}
	where the right-hand side comes out from \eqref{repRemarkDriftCoef}. The equation  \eqref{solBCdriftProp} can be associated with $r(t)$ in Proposition \ref{propSolBCgeneral} with $c_0=a > b=0$. In particular, $u(t,0)$ coincides with 
	\begin{align*}
	r(t)= E_\frac{1}{2}(-c \sqrt{t}) = \mathbf{E}_0[e^{-c\, L_t}]
	\end{align*}
	which is the Mittag-Leffler function introduced in \eqref{BChelpful2}.
	
	We also notice that $\mathbf{P}_0(L_0 < 0)=1-\mathbf{P}_0(L_0=0) = 0$ and $\mathbf{1}_{(L_0\geq x)} \to 1$ as $x\to 0^+$. On the other hand, for $t>0$, as $x \to 0^+$, $\mathbf{P}_0(L_t < x)\to 0$ and $\mathbf{1}_{(L_t \geq x)} \to 1$. 
	\item[-]
	\emph{Null elastic coefficient}. For 
	$$\mu \geq 0, c = 0$$ 
	we obtain that
	\begin{align*}
	P_t^\mu \bar{\mathbf{1}}(x) = 1  , \quad \forall\, x.
	\end{align*}
	The boundary behaviour is characterized by
	\begin{align}
	\label{BCreflectionCASEI}
	\mathfrak{D}^{\frac{1}{2},\mu}_t u(t,0) + \mu\, u(t,0) = \mu, \quad t > 0
	\end{align}
	for which the function $u(t,0)$ can be associated with $r(t)$ in Proposition \ref{propSolBCgeneral} with $a=b=\mu$. It follows that
	\begin{align*}
	r(t) = 1 \quad \forall\, t \geq 0
	\end{align*}
	coincides with $u(t,0)$, $t\geq 0$.
	
	The lifetime $\zeta^\mu$ of the process is infinite almost surely. Indeed, $\forall\, x \in [0, \infty)$, $\mathbf{P}_x(\zeta^\mu > t)=1$ for any $t\geq 0$. The multiplicative functional $M^\mu_t=\mathbf{1}_{(t < \infty)}$ is associated with reflection at $x=0$ of the drifted Brownian motion. 
	\item[-] Let us consider
	$$c \in (0, \infty).$$ 
	We have  $\mu + c =a > b =\mu$. Then, from Proposition \ref{propSolBCgeneral}, $a-b=c > 0$ implies that $0< r(t) < 1$ $\forall\, t>0$. In particular, $u(t,0)$ coincides with
	\begin{align}
	\label{caseElliptic}
	r(t) = \frac{\mu}{\mu + c} \mathbf{P}_0(L_t \geq T_{\mu + c}) + \mathbf{P}_0(L_t < T_{\mu+c}).
	\end{align}
	Equivalently, by using representation \eqref{solmuSemigRep4} we have that
	\[
	r(t) = \mathbf P_0(L_t \wedge T_\mu < T_{c})
	\]
	\item[-] Let us consider 
	$$c \to \infty \quad \textrm{with} \quad \mu \geq 0.$$ Obviously, this special case does not completely agree with the inital datum. Formulas \eqref{solmuSemigRep} and \eqref{solmuSemigRep2} take the form
	\begin{align*}
	u(t,x) = 1- e^{-x \mu }\, \mathbf{P}_0(L_t \geq x) = 1- P(L_t \wedge T_\mu \geq x)
	\end{align*}
	for which $u(t,0)=1 - \mathbf{P}_0 (L_t \geq 0) = 0$. The formal limit in \eqref{BCprobprobExit} gives the Dirichlet boundary condition. This corresponds to the fact that, by exploiting the representation \eqref{solmuSemigRep3}, we obtain
	\[
	u(t,x) = 1 - \mathbf P_0\left(\max_{0 \leq s \leq t} X^{-\mu}_s \geq x\right) = \mathbf P_x\left(\min_{0 \leq s \leq t} X^{\mu}_s > 0\right) = \mathbf P_x(\zeta^\mu > t)
	\]	
	where $ \zeta^\mu $ now represents the lifetime of a drifted Brownian motion with an absorbing barrier at zero.
\end{itemize}

\subsection{The negatively drifted Brownian motion}

We focus on the function 
$$u \in C^{1,2}(AC(0, \infty)\times [0, \infty), [0, \infty))$$ solving 
\begin{equation}
\label{probExit-neg}
\begin{cases}
\displaystyle  \frac{\partial u}{\partial t} = -\mu \frac{\partial u}{\partial x} + \frac{\partial^2 u}{\partial x^2}\\
\displaystyle u(0, x) = \mathbf{1}(x \geq 0)
\end{cases}
\end{equation}
with the boundary condition
\begin{align}
\label{BCprobprobExit-neg}
\mathfrak{D}^{\frac{1}{2},\mu}_t u(t, 0)  + c \, u(t,0) = 0, \quad t > 0
\end{align}
%
%
where $ c $ plays now the role of the elastic coefficient in the condition
\begin{equation}
\label{BCstandard-neg}
\frac{\partial u}{\partial x}(t,0) = c\, u(t,0), \quad t>0.
\end{equation} 
Formally, the relation between the conditions \eqref{BCprobprobExit-neg} and \eqref{BCstandard-neg}
is justified by the equation
\begin{align*}
\mathfrak{D}^{\frac{1}{2},\mu}_t \ell(t,x) = - \frac{\partial}{\partial x} \ell(t,x)\quad t>0, \, x>0.
\end{align*} 
corresponding to \eqref{probell} in the tempered stable case.

For the negatively drifted Brownian motion $\{\widetilde{X}^{-\mu}_t\}_{t\geq 0}$ with generator $(G_{-\mu}, D(G_{-\mu}))$ we write
\begin{align}
\label{solmuSemig-neg}
P_t^{-\mu} f(x) = \mathbf{E}_x[f(\widetilde X^{-\mu})]
\end{align}
which is the analogue of \eqref{solmuSemig}.

\begin{theorem}
\label{thmII}
Let us consider $u=u(t,x)$ given in \eqref{solmuSemig-neg}. Then, the following statements hold:
\begin{itemize}
\item[i)] $u$ is the unique (classical) solution to \eqref{probExit-neg} - \eqref{BCprobprobExit-neg};
\item[ii)] $u$ is the unique (classical) solution to \eqref{probExit-neg} - \eqref{BCstandard-neg};
\item[iii)] $u$ has the probabilistic representation
\begin{align*}
u(t,x) = 1 - \mathbf{E}_0 \left[ \left( 1 - e^{- c\,(L_t - x) } \right) \mathbf{1}_{(L_t > x)} \right]
\end{align*}
where $L$ is an inverse to a relativistic stable subordinator with symbol \eqref{symbGenSub}, $\eta=(-\mu/2)^2$.
\end{itemize}
\end{theorem}
\begin{proof}
Point iii). By exploiting \eqref{eq:app-res} in the Appendix with $-\mu$ in place of $\mu$ we get
\begin{align}\label{eq:lap-sol-mu-neg}
\widetilde{u}(\lambda, x) &=  R_\lambda \mathbf{ \bar 1} (x)
 =  \frac{1}{\lambda} - \frac{1}{\lambda} \left( 1-  \frac{\Phi(\lambda)}{c  + \Phi(\lambda)} \right) e^{-x \Phi(\lambda)}
\end{align}
where, as usual, $\Phi = \Phi(\lambda)= \sqrt{\lambda + \mu^2/4} - \mu/2$. Since
\begin{align*}
\frac{e^{-x\Phi(\lambda)}}{\lambda} = \int_0^\infty e^{-\lambda t} \mathbf{P}_0(L_t > x)\, dt
\end{align*}
and
\begin{align*}
\frac{\Phi(\lambda)}{\lambda} \frac{e^{-x \Phi(\lambda)}}{\Phi(\lambda) + c} 
= &
\frac{\Phi(\lambda)}{\lambda} \int_0^\infty e^{-cw} e^{-(x+w) \Phi(\lambda)} dw 
\\
= & \int_0^\infty e^{-\lambda t} \left( \int_0^\infty e^{-cw} \ell(t, x+w)\, dw \right) dt\\
= & \int_0^\infty e^{-\lambda t} \mathbf{E}_0\left[ e^{-c (L_t - x)} \, \mathbf{1}_{(L_t >x)} \right] dt
\end{align*}
the solution $u$ can be written as
\begin{align*}
u(t,x) = 1 - \mathbf{E}_0[\mathbf{1}_{(L_t >x)}] + \mathbf{E}_0\left[ e^{-c (L_t - x)} \, \mathbf{1}_{(L_t >x)} \right]
\end{align*}
that is
\begin{align*}
u(t,x) = 1 - \mathbf{E}_0 \left[ \left(1 - e^{-c (L_t -x)} \right) \mathbf{1}_{(L_t >x)} \right].
\end{align*}
Then we get the result claimed in iii).

In order to prove point ii) we check that $ u(t, 0) $ satisfies \eqref{BCstandard-neg}. In fact from \eqref{eq:el-res0} with drift $ -\mu $ we get
\begin{equation}\label{eq:lap-sol-mu-neg0}
	\tilde u (\lambda, 0)  = 
	\frac 1{(\sqrt{\lambda + \mu^2/4} + \frac \mu 2 )(\sqrt{\lambda + \mu^2/4} - \frac \mu 2 + c)} 
	= 
	\frac{1}{(\Phi + \mu)(\Phi + c)}
\end{equation}

and then from \eqref{symbDermu}, the Laplace transform of the LHS of \eqref{BCprobprobExit-neg} is 
\begin{align}
\label{lapUniqM-neg}
\Phi  \tilde u(\lambda, 0) - \frac{\Phi}{\lambda} u(0, 0) + c \tilde u(\lambda, 0) &=
\frac{\Phi}{(\Phi + \mu)(\Phi + c)} - \frac{\Phi}{\lambda} 
+ \frac{c}{(\Phi + \mu)(\Phi + c)}
\\
&=
\frac{1}{\Phi + \mu} - \frac{1}{\Phi + \mu} = 0. \notag 
\end{align}
verifying  that $ u(t,0) $ satisfies  \eqref{BCprobprobExit-neg}.

Point i) does not need to be proved. This concludes the proof. 
\end{proof}

As in Corollary \ref{coro:equivI},  also in this case we are able to show that $\overline{M}^{-\mu}_t$ is equivalent to $M^{-\mu}_t$ associated with $X^{-\mu}$ and therefore, with the generator $(G_{-\mu}, D(G_{-\mu}))$. In particular, $\overline{M}^{-\mu}_t$ is uniquely characterized by the boundary condition \eqref{BCprobprobExit-neg}.

Below we present the last result for the negatively drifted Brownian motion. 

\begin{theorem}
\label{last-thmII}
The solution to \eqref{probExit-neg} - \eqref{BCprobprobExit-neg} has the probabilistic representation
\begin{align*}
\displaystyle u(t,x) = 1 - \mathbf{E}_0 \left[ \left( 1- e^{-c\, (S^{\mu}_t - x)} \right) \mathbf{1}_{(S^{\mu}_t > x)} \right]
\end{align*}
where
\begin{align*}
S^{\mu}_t = \max_{0\leq s \leq t} X^{\mu}_s, \quad t>0,\; \mu>0.
\end{align*}
\end{theorem}
\begin{proof}
The proof follows immediately from Theorem \ref{thm0}.
\end{proof}

\section{Conclusion}
We observe that Theorem \ref{last-thmII} is the analogue to Theorem \ref{last-thmI}. Such results give clear representations of the solutions, in both cases, in which we have positive or negative drift,
\begin{align*}
u(t,x) = 1 - \mathbf{E}_0 \left[ \left( 1 - e^{- c (S^{\pm \mu}_t -x)} \right), S^{\pm \mu}_t > x \right], \quad t>0,\; x >0.
\end{align*}
Moreover, in our view, Theorem \ref{thmIbis} and Theorem \ref{thmII} seems to be quite interesting with regards to the applications. Indeed, they are written in terms of very simple processes, that is, non-decreasing processes on $(0, \infty)$. Formula \eqref{resultGravShir} suggests also a representation in terms of the local time which is usually sneaky. Some fruitful applications of such representations may arise in numerical solutions, optimization, inverse problems and so forth. Indeed, in  these contexts, it is important to obtain fast and accurate simulations. On the other hand, the proposed algorithms may result in high demanding computational tasks, as for the Monte Carlo approximations for instance. For a description of simulation algorithms for a Brownian motion on the half-line with boundary conditions the interested reader can consult \cite{sticky-numerical} and references therein.  Such algorithms are based on spatial discretizations for the generator of the process. Clearly our results provide a simpler and immediate alternative which only requires the simulation of the increments of a tempered subordinator, which is a straightforward task (see e.g. \cite{MeeSabChe}).

\section{Appendix}
\subsection{Elastic Brownian motion with drift.}
Consider EDBM $ \widetilde X^{\mu} $, $ \mu \in \mathbb R $
(at this stage we do not make the distinction between positive or negative drift yet).
 Let $ \hat X^\mu $ be a reflecting Brownian motion with drift.
The process $ \widetilde X^{\mu} $ can be represented in the following way
%
\begin{equation}
	\widetilde {X}^\mu_t = 
	\begin{cases}
		\hat X^\mu_t & t \leq \zeta^\mu \\
		\dagger & t > \zeta^\mu 
	\end{cases}
\end{equation}
where $ \zeta^\mu  = \inf\{ t: \gamma_t > T_c\}$ is the lifetime of the process, $ T_c $ is an independent 
exponential random variable with parameter $ c $ and $ c > 0 $ denotes the elastic coefficient.
The distribution of $ \widetilde X^\mu $ can be obtained as follows. First consider the case where the starting point
$ x=0 $.  Then 
\begin{align*}
	P_0(\widetilde X^\mu_t \in A) &= \int_A p(t, 0, y) \mathrm d y = 
	P_0(\hat X^\mu_t \in A, \zeta^\mu > t) \\
	&=
	\mathbf E_0 \left[ 1_{(\hat X^\mu_t \in A)} P(\zeta^\mu > t|\mathcal F_t)\right]
	\\
	&=
	\mathbf E_0 \left[ 1_{(\hat X^\mu_t \in A)} e^{-c \gamma_t}\right] \\
	&=
	\int_{A} \int_{0}^{\infty} e^{- c v} P(\hat X^\mu_t \in \mathrm d y, \gamma_t \in \mathrm d v) \\
	&=
	\int_{A} \int_{0}^{\infty} e^{- c v} P(S^{-\mu}_t - B^{-\mu}_t \in \mathrm d y, S^{-\mu}_t \in \mathrm d v) \\
\end{align*}
where in the last step we use relation 
\begin{equation*}
	(\hat X^{\mu}, \gamma^\mu ) \stackrel{d}{=} (S^{-\mu} - B^{-\mu}, S^{-\mu})
\end{equation*}
from \cite{GravShir}, where $ S^{-\mu} $ now denotes the maximum of the Brownian motion with drift $ B^{-\mu} $. By using the explicit expression of the joint distribution of $ (B^{-\mu}, S^{-\mu}) $ 
(see e.g. \cite{Shep79})
one has
\begin{align}\label{eq:elastic-dens0}
	p(t, 0, y) &=
	2e^{\frac \mu 2 y - \frac{\mu^2}{4}t} 
	\left[
	g(t, y) - \left(\frac \mu 2 + c \right)
	\int_{0}^{\infty} e^{-\left(\frac \mu 2 + c \right)v}
	g(t, v+y) \,\mathrm d v
	\right] \notag 
\end{align}
Finally for $ x > 0, y > 0 $, by the Markov property of $ \widetilde  X^\mu $ we have that
\begin{align}
	P_x(\widetilde X_t^\mu \in \mathrm d y) &= 
	P_x(X_t^\mu \in \mathrm d y, \tau_0^\mu > t) +  
	P_x(\widetilde  X^\mu \in \mathrm d y, \tau_0^\mu > t)
\end{align}
then 
\[
p(t, x, y)  = 
	\bar p (t, x, y)  + 
\int_{0}^{t} f^{-\mu}(-x, 0, s) p(t-s, 0, y) \,\mathrm d s  
\]
where $ \bar p $ denotes the density on $(0, \infty)$ of a killed Brownian motion with drift
and $ f^{-\mu} $ is the density of the first passage time through 0 of a Brownian motion with drift $ -\mu $ 
and starting point $ -x $ (for details on this last step see \cite{Iafrate19} formula (28)).

The following results hold for the EDBM.
\begin{proposition}
	\begin{enumerate}[(i)]
		\item The Green function of the elastic Brownian motion with drift reads
		\begin{equation}\label{eq:el-green}
			G_\lambda(x, y) = 
			\left\{ 
			\begin{aligned}
				\frac{1}{2\Lambda} e^{- (\frac \mu 2+ \Lambda)x} \left[
				e^{(\frac \mu 2 + \Lambda)y} + \frac{\Lambda - \frac \mu 2- c}{\Lambda + \frac \mu 2 + c}
				e^{(\frac \mu 2 - \Lambda)y}
				\right] & \qquad x> y\\
				\frac{1}{2\Lambda} e^{(\frac \mu 2 - \Lambda)y}  \left[
				e^{(\Lambda - \frac \mu 2)x} + \frac{\Lambda - \frac \mu 2- c}{\Lambda + \frac \mu 2 + c}
				e^{-(\frac \mu 2 + \Lambda)x}
				\right] & \qquad  x \leq y.
			\end{aligned}
			\right. 
		\end{equation}
		where $ \Lambda = \sqrt{\lambda + \frac{\mu^2}{4}} $.
		
		\item  The resolvent of the elastic Brownian motion with drift reads
		\begin{align}\label{eq:el-resolvent}
			R_\lambda f(x) &= 
			%
			\frac{1}{2\Lambda} \left[
			e^{- (\frac \mu 2+ \Lambda)x} 
			\int_{0}^{x} 
			e^{(\frac \mu 2 + \Lambda)y} 
			f(y) \mathrm d y 
			 + e^{(\Lambda - \frac \mu 2)x}
			 \int_{x}^{\infty}  e^{(\frac \mu 2 - \Lambda)y} f(y) \mathrm dy
			\right]
			\\ 
			& \qquad +
			\frac{1}{2\Lambda} e^{-(\frac \mu 2 + \Lambda)x}
			 \frac{\Lambda - \frac \mu 2- c}{\Lambda + \frac \mu 2 + c}
			\int_{0}^{\infty}  e^{(\frac \mu 2 - \Lambda)y} f(y) \mathrm dy 
			\notag 
		\end{align}
		
		\item  
		The right limit of the resolvent at zero is  
		\begin{equation}\label{eq:el-res0}
			R_\lambda f(0^+) = 
			\frac{1}{\Lambda + \frac \mu 2 + c} \int_{0}^{\infty} e^{(\frac{\mu}{2} - \Lambda)y} f(y) \,\mathrm d y
		\end{equation}
	
	\end{enumerate}
\end{proposition}

\begin{proof}
	
	By taking the $ \lambda -$Laplace transform of \eqref{density:pmu}, and by using
	the formula
	\[
	\int_{0}^{\infty} e^{-\lambda t} g(t, x) \,\mathrm d t = \frac{1}{2\sqrt \lambda}
	e^{-\sqrt{\lambda} |x|}
	\]
	one has
	\begin{align*}
		G_\lambda(x, y) &= 
		\frac{e^{\frac{\mu}{2}(y-x)}}{2\Lambda}
		\left[
		e^{- \Lambda |x-y|} + 	e^{- \Lambda (x+y)} 
		-2\left(\frac \mu 2 + c\right) \int_{0}^{\infty} 
		e^{-\left(\frac \mu 2 + c \right)v} e^{- \Lambda (x + y + v)} \,\mathrm d v
		\right]
		\\
		&=
		\frac{e^{\frac{\mu}{2}(y-x)}}{2\Lambda}
		\left[
		e^{- \Lambda |x-y|} + 	e^{- \Lambda (x+y)} 
		\left(1 - \frac{2\left(\frac \mu 2 + c\right)}{\Lambda + \frac \mu 2 + c}\right)
		\right]
		\notag \\
		&=
		\frac{e^{\frac{\mu}{2}(y-x)}}{2\Lambda}
		\left[
		e^{- \Lambda |x-y|} + 	e^{- \Lambda (x+y)} 
		\frac{\Lambda - \frac \mu 2 - c}{\Lambda + \frac \mu 2 + c}
		\right]
	\end{align*}
	which can be readily rearranged into \eqref{eq:el-green}.
	
	Formula \eqref{eq:el-resolvent} can be obtained by integrating \eqref{eq:el-green} 
	in the following way
	\begin{align}\label{eq:el-resolvent1}
	R_\lambda f(x) &= \int_{0}^{\infty} G_\lambda(x, y) f(y) \,\mathrm d y 
	\\
	&=
	\frac{1}{2\Lambda} e^{- (\frac \mu 2+ \Lambda)x} 
	\int_{0}^{x} \left[
	e^{(\frac \mu 2 + \Lambda)y} + \frac{\Lambda - \frac \mu 2- c}{\Lambda + \frac \mu 2 + c}
	e^{(\frac \mu 2 - \Lambda)y}
	\right] f(y) \mathrm d y 
	\notag \\ 
	& \qquad +
	\frac{1}{2\Lambda} \left[
	e^{(\Lambda - \frac \mu 2)x} + \frac{\Lambda - \frac \mu 2- c}{\Lambda + \frac \mu 2 + c}
	e^{-(\frac \mu 2 + \Lambda)x}
	\right]
	\int_{x}^{\infty}  e^{(\frac \mu 2 - \Lambda)y} f(y) \mathrm dy 
	\notag 
	\end{align}
	which can be easily simplified into \eqref{eq:el-resolvent}. Finally by taking the limit $ x \to 0^+ $ in \eqref{eq:el-resolvent1} one has
	\[
	R_\lambda f(0^+) = 
	\frac{1}{2\Lambda} \left[ 1 + \frac{\Lambda - \frac \mu 2- c}{\Lambda + \frac \mu 2 + c} \right]
	\int_{0}^{\infty}  e^{(\frac \mu 2 - \Lambda)y} f(y) \mathrm dy 
	\]
	and the result follows.
\end{proof}

\begin{proposition}
	The resolvent \eqref{eq:el-resolvent} applied to $ f = \mathbf{\bar 1} $ reads
	\begin{equation}\label{eq:app-res}
		R_\lambda \mathbf {\bar 1} (x) = 
		\frac 1 \lambda  - 
		\frac 1 \lambda \left( 1 - \frac{\Lambda + \frac \mu 2 }{\Lambda + \frac \mu 2 + c} \right) 
		e^{-(\frac \mu 2 + \Lambda)x}.
	\end{equation}
\end{proposition}

\begin{proof}
	
Note that 
\begin{equation}\label{eq:lambda-eq}
	\left(\Lambda + \frac \mu 2\right)\left(\Lambda - \frac \mu 2\right) = \left(\sqrt{\lambda + \frac{\mu^2}4} - \frac \mu 2 \right)\left(\sqrt{\lambda + \frac{\mu^2}4} + \frac \mu2 \right)
	= \lambda.
\end{equation}

By substituting $ f = \mathbf{\bar 1} $ in \eqref{eq:el-resolvent} we have
\begin{align}
	R_\lambda \mathbf {\bar 1} (x) &=
	\frac{e^{- (\frac \mu 2+ \Lambda)x} }{2\Lambda(\Lambda + \frac \mu 2)}
	\left( e^{(\frac \mu 2 + \Lambda )x} - 1 \right)
	+ 
	\frac{1}{2 \Lambda (\Lambda - \frac \mu 2) }
	+
	\frac{1}{2\Lambda} e^{-(\frac \mu 2 + \Lambda)x}
	\frac{\Lambda - \frac \mu 2- c}{\Lambda + \frac \mu 2 + c} \cdot 
	\frac{1}{\Lambda - \frac \mu 2} 
	\notag 
	\\
	&=
	\frac 1 {2\Lambda} \left(\frac{1}{\Lambda + \frac \mu 2}  + \frac{1}{\Lambda - \frac \mu 2} \right)
	+ 
	\frac 1 {2\Lambda} \left(-\frac{1}{\Lambda + \frac \mu 2}  + \frac{1}{\Lambda - \frac \mu 2}
	\frac{\Lambda - \frac \mu 2- c}{\Lambda + \frac \mu 2 + c}
	\right) e^{-(\frac \mu 2 + \Lambda)x}
	\notag \\
	&=
	\frac 1 \lambda 
	+ \frac 1 {2\Lambda} 
	\left(
	-\frac{1}{\Lambda + \frac \mu 2}  - \frac{1}{\Lambda - \frac \mu 2} 
	+ 
	\frac{1}{\Lambda - \frac \mu 2} \cdot \frac{2\Lambda}{\Lambda + \frac \mu 2 + c}
	\right) e^{-(\frac \mu 2 + \Lambda)x}
	\notag \\
	&=
	\frac 1 \lambda  - 
	\frac 1 \lambda \left( 1 - \frac{\Lambda + \frac \mu 2 }{\Lambda + \frac \mu 2 + c} \right) 
	e^{-(\frac \mu 2 + \Lambda)x}
	\notag 
\end{align}

\end{proof}

\subsection{Proof of the statements in Section \ref{helpfulSection}}

\begin{proof}[Proof of the formula \eqref{helpfulProof}]
Since
\begin{align*}
\int_0^\infty e^{-\lambda t} \, g(t,x-y)\, dt = \frac{1}{2} \frac{e^{-|x-y| \sqrt{\lambda}}}{\sqrt{\lambda}} \quad \textrm{and} \quad \int_0^\infty e^{-\lambda t} \, g(t,x+y)\, dt = \frac{1}{2} \frac{e^{-(x+y) \sqrt{\lambda}}}{\sqrt{\lambda}}
\end{align*}
we have that
\begin{align}
\label{proof1part1}
\int_0^\infty e^{-\lambda t} \int_0 \big( g(t, x-y) - g(t,x+y) \big) dy \, dt = \frac{1}{\lambda} - \frac{1}{\lambda} e^{-x\sqrt{\lambda}}
\end{align}
Now we note that
\begin{align*}
\frac{x}{s}g(s,x) = \frac{x}{s} \frac{e^{-\frac{x^2}{4s}}}{\sqrt{4\pi s}} = - 2 \frac{\partial}{\partial x} \frac{e^{-\frac{x^2}{4s}}}{\sqrt{4 \pi s}}, \quad x \in [0, \infty), \; s>0 
\end{align*}
for which we have 
\begin{align}\label{eq:lap-gauss}
\int_0^\infty e^{-\lambda s} \, g(s,x) \, ds = \frac{1}{2}\frac{e^{-x \sqrt{\lambda}}}{\sqrt{\lambda}}
\end{align}
and
\begin{align*}
\int_0^\infty e^{-\lambda s}\, \frac{x}{s}g(s,x)  \, dx =  -\, \frac{\partial}{\partial x} \frac{e^{-x \sqrt{\lambda}}}{\sqrt{\lambda}} =  e^{-x \sqrt{\lambda}} .
\end{align*}
Then, 
\begin{align}
\label{proof1part2}
\int_0^\infty e^{-\lambda t} \left( 1 - \int_0^t \frac{x}{s} g(s,x)\, ds \right) dt = \frac{1}{\lambda} - \frac{1}{\lambda} e^{-x \sqrt{\lambda}}
\end{align}
By comparing \eqref{proof1part1} with \eqref{proof1part2} we prove that, $\forall\, x$,
\begin{align*}
P^0_t \mathbf{1}(x) = F(t,x).
\end{align*}
Moreover, for our convenience, we observe that
\begin{align*}
2 \int_0^\infty e^{-\lambda t}  \int_0^\infty e^{-c_0 \, w} g(t,w+x) \, dw\, dt 
= & \int_0^\infty e^{-c w} \frac{e^{-(w+x)\sqrt{\lambda}}}{\sqrt{\lambda}} \, dw\\
= & \frac{e^{-x \sqrt{\lambda}}}{\lambda} \frac{\sqrt{\lambda}}{c + \sqrt{\lambda}}
\end{align*}
and therefore
\begin{align}
\int_0^\infty e^{-\lambda t}\, P^0_t \mathbf{1}(x)\, dt = \frac{1}{\lambda} - \frac{e^{-x \sqrt{\lambda}}}{\lambda} \frac{c}{c + \sqrt{\lambda}}.
\end{align}
\end{proof}

\begin{proof}[Proof of the formula \eqref{helpfulProbRep}]

Let us consider the non-negative and non-decreasing process $A_t$, $t\geq 0$ with probability
\begin{align*}
\mathbf{P}_0(A_t > x) = 2 \int_x^\infty g(t,s)\,ds .
\end{align*}
Consider the inverse $A^{-1}_t$, $t\geq 0$ which is, by construction, a non-negative and non-decreasing process. By definition we have that
\begin{align*}
\mathbf{P}_0(A^{-1}_x < t) = \mathbf{P}_0(A_t > x)
\end{align*}
with 
\begin{align*}
\frac{\partial}{\partial t} \mathbf{P}_0(A^{-1}_x < t) = 2 \int_x^\infty \frac{\partial^2}{\partial s^2} g(t,s)\, ds = \frac{x}{t} g(t,x). 
\end{align*}
Since $\mathbf{P}_0(A^{-1}_0 < t) = \mathbf{P}_0(A_t > 0)=1$ and $A_t$ is continuous we obtain $F(t,0)=\mathbf{E}_0[e^{-c\, A_t}]$.
\end{proof}

\begin{proof}[Proof of the formula \eqref{BChelpful1}]
It holds that
\begin{align*}
\frac{\partial F}{\partial x} (t,x) \bigg|_{x=0} = c_0 \, F(t,x) \bigg|_{x=0}.
\end{align*}
Indeed, $F$ is the density law of the elastic Brownian motion on $[0, \infty)$.
\end{proof}

\begin{proof}[Proof of the formula \eqref{BChelpful2}]
Since $L^0_t$ is the inverse to a stable subordinator, the density $\ell(t,x)= 2\, g(t,x)$ is such that 
\begin{align*}
D^\frac{1}{2}_t \ell = - \frac{\partial \ell}{\partial x} 
\end{align*}
and the potential
\begin{align*}
\hat{\ell}(t, c) = \int_0^\infty e^{-c \,w} \ell(t,w)\, dw = E_\frac{1}{2}(- c\, \sqrt{t}) 
\end{align*}
is the Mittag-Leffler function. It is well-known that the Mittag-Leffler is an eigenfunction for the Caputo derivative. That is,
\begin{align*}
D^\frac{1}{2}_t \hat{\ell} = - c \, \hat{\ell}.
\end{align*}
By observing that $\hat{\ell}(t,c)= F(t,0)$ we get formula \eqref{BChelpful2}. The problem to check that
\begin{align*}
D^\frac{1}{2}_t F(t,x) \bigg|_{x=0} = - c \, F(t,x) \bigg|_{x=0}
\end{align*}
is part of the results presented in this work.
\end{proof}

\end{document}